\documentclass[11pt]{amsart}
\usepackage[T1]{fontenc}
\usepackage{lmodern}
\usepackage{amsmath,amssymb,amsthm,mathtools,mathrsfs}
\usepackage{microtype}
\usepackage[hidelinks]{hyperref}
\newtheorem{theorem}{Theorem}[section]
\newtheorem{lemma}[theorem]{Lemma}
\newtheorem{proposition}[theorem]{Proposition}
\newtheorem{corollary}[theorem]{Corollary}
\theoremstyle{definition}

\theoremstyle{remark}
\newtheorem{remark}[theorem]{Remark}
\numberwithin{equation}{section}
\newcommand{\R}{\mathbb R}
\newcommand{\HH}{\mathcal H}

\DeclareMathOperator{\co}{co}

\DeclareMathOperator{\relint}{relint}

\DeclareMathOperator{\diam}{diam}

\DeclareMathOperator{\bary}{bar}
\DeclareMathOperator{\ext}{ext}
\DeclareMathOperator{\tr}{tr}

\title[Transport to non-convex polyhedral domains]{Optimal transport to non-convex polyhedral domains}

\author{Shibing Chen}
\address{}
\email{chenshibing1982@hotmail.com}

\author{Qi-Rui Li}
\address{School of Mathematical Sciences, Zhejiang University, Hangzhou 310058, China}
\email{qi-rui.li@zju.edu.cn}

\author{Yuanyuan Li}
\address{Institute for Theoretical Sciences, Westlake University, Hangzhou, 310030, China}
\email{lyyuan@westlake.edu.cn}

\author{Jiakun Liu}
\address{School of Mathematics and Statistics, The University of Sydney, Camperdown, NSW 2006, Australia}
\email{jiakun.liu@sydney.edu.au}

\author{Xiaotian Wu}
\address{Institute for Theoretical Sciences, Westlake University, Hangzhou, 310030, China}
\email{wuxiao-tian@westlake.edu.cn}

\date{\today}
\begin{document}
\begin{abstract}
We study optimal transport with a non-convex polyhedral target in $\R^n$. The closure of the singular set is locally a smooth hypersurface outside a relatively closed set of Hausdorff dimension at most $n-2$. For partial transport between domains separated by a hyperplane, the source free boundary has the same regularity.
\end{abstract}
\maketitle
\baselineskip=16.4pt
\parskip=3pt

\section{Introduction}

In this paper we study optimal transport and optimal partial transport with constant densities when the target is a non-convex polyhedral domain in $\R^n$. Let $\Omega,\Omega^*\subset\R^n$ be bounded open sets, $n\ge2$, and let $0<m\le\min\{|\Omega|,|\Omega^*|\}$. A transport plan of mass $m$ is a nonnegative Borel measure $\gamma$ such that
\[
\gamma(\R^n\times\R^n)=m,\qquad
(\pi_1)_\#\gamma\le\chi_\Omega\,dx,\qquad
(\pi_2)_\#\gamma\le\chi_{\Omega^*}\,dy.
\]
An optimal plan minimizes
\begin{equation*}
\int_{\R^n\times\R^n}\tfrac12|x-y|^2\,d\gamma(x,y).
\end{equation*}
When $m=|\Omega|=|\Omega^*|$, by Brenier's theorem~\cite{Br1} there exists a convex function $u$ satisfying
\begin{equation}\label{transport}
(Du)_\#\chi_\Omega=\chi_{\Omega^*}.
\end{equation}
We write
\begin{equation*}
A=\{x\in\Omega:u\text{ is not differentiable at }x\},
\qquad \Sigma=\overline A\cap\Omega.
\end{equation*}
In particular, $\Sigma$ may contain differentiability points of $u$.

The regularity theory for convex targets was developed by Caffarelli~\cite{C92,C96}. Classical boundary regularity for smooth uniformly convex domains was established by Delano\"e~\cite{D91} in dimension two and Urbas~\cite{U1} in arbitrary dimensions. Chen--Liu--Wang~\cite{CLW1} obtained global boundary estimates for convex domains with $C^{1,1}$ boundaries. For arbitrary bounded convex planar domains with constant densities, Savin--Yu~\cite{SY} proved global $W^{2,p}$ estimates for every finite $p$. In arbitrary dimensions, Collins--Tong~\cite{CT} obtained global $W^{2,p}$ estimates for positive Dini continuous densities, and Chen--Li--Wang~\cite{CLiW} proved these estimates for positive continuous densities, without boundary smoothness assumptions. When the domains are $C^{1,\alpha},$ and densities are positive and $C^\alpha,$ Collins--Tong~\cite{CT} established the global $C^{2,\alpha}$ estimate.

Singularities may occur when the target is non-convex. In the plane, Yu~\cite{Yu} studied the structure of the singular set under some mild conditions on the source and target, and Figalli~\cite{AFi2} investigated the case with general non-convex domains. Figalli--Kim~\cite{FK1} proved partial regularity in arbitrary dimensions without convexity assumptions on the domains. Their result shows that points satisfying $\partial u(x)\cap\Omega^*\ne\varnothing$ belong to an open set on which the potential is regular. Kitagawa--McCann~\cite{KM1} established structure and stability results for discontinuities associated with separated target components. Chen--Liu~\cite{CL3} proved $C^{1,\beta}$ regularity of the interface for two separated convex target components without assuming convexity of the source. They also proved higher regularity for smooth uniformly convex target components and H\"older continuous densities.

For a polygonal target, Chen--Li--Liu~\cite{CLL} proved that the planar singular set is locally a smooth curve outside finitely many points. They also formulated the corresponding higher-dimensional conjecture for $\Sigma$ in \cite[Conjecture~5.1]{CLL}. The general convex-function stratification theorem of Alberti--Ambrosio--Cannarsa~\cite[Theorem~4.1]{AAC} controls points at which the subdifferential has dimension at least two. It remains to bound the Hausdorff dimension of the sets of points whose subdifferentials are singletons or segments with an extreme point on a lower-dimensional face of $\Omega^*$. The finiteness of the faces alone does not control a bound on the Hausdorff dimension of their images under $Dv$.

Here a polyhedral domain is the interior of a finite union of closed full-dimensional polytopes. Fix a finite face decomposition of its boundary. We call the union of faces of dimension at most $n-2$ the target skeleton. Near an open boundary facet, the domain is a halfspace. Near any boundary point, it is a finite union of full-dimensional convex polyhedral cones.

Polygonal and polyhedral domains arise naturally in computational geometry and numerical approximation, where complicated shapes are represented by polygonal or polyhedral meshes. For non-convex targets, the optimal transport map may be discontinuous. The following theorem describes the singular set for polyhedral targets in arbitrary dimensions.

\begin{theorem}\label{complete}
Suppose that $\Omega$ is bounded and open, and that $\Omega^*$ is a bounded polyhedral domain with $|\Omega|=|\Omega^*|$. Let $u$ be a convex solution of \eqref{transport}. There is a relatively closed set $E\subset\Sigma$, with $\dim_{\HH}E\le n-2$, such that $\Sigma$ is locally a smooth $(n-1)$-dimensional hypersurface at every point of $\Sigma\setminus E$.
\end{theorem}

For $0<m<\min\{|\Omega|,|\Omega^*|\}$, let $U\subset\Omega$ and $V\subset\Omega^*$ be the active source and target, and put $\mathcal F=\partial U\cap\Omega$. The existence, uniqueness and free boundary geometry of partial transport with separated supports were established by Caffarelli--McCann~\cite{CM}. Figalli~\cite{AFi} extended the result to the case of overlapping source and target domains. Further free boundary regularity was obtained by Indrei~\cite{I} and Chen--Liu--Wang~\cite{CLW3} under convexity assumptions on the fixed domains. For a polygonal target, Chen--Li--Liu~\cite{CLL} proved smoothness of the source free boundary outside finitely many points and formulated \cite[Conjecture~5.2]{CLL} in higher dimensions.

\begin{theorem}\label{partial}
Suppose that $\Omega$ is bounded and open, and that $\Omega^*$ is a bounded polyhedral domain separated from $\Omega$ by a hyperplane. Let $0<m<\min\{|\Omega|,|\Omega^*|\}$. There is a relatively closed set $E_{\mathcal F}\subset\mathcal F$, with $\dim_{\HH}E_{\mathcal F}\le n-2$, such that $\mathcal F$ is locally a smooth hypersurface at every point of $\mathcal F\setminus E_{\mathcal F}$.
\end{theorem}

\begin{remark}
Theorems~\ref{complete} and~\ref{partial} also hold for the source and target measures $f\chi_\Omega\,dx$ and $g\chi_{\Omega^*}\,dy$, provided
\[
0<\lambda\le f,g\le\Lambda,\qquad f\in C^\infty(\Omega),
\]
and $g$ extends to a positive smooth function on a neighbourhood of $\overline{\Omega^*}$.
\end{remark}

These two theorems resolve the conjectures posed in our earlier work \cite{CLL}. No convexity or boundary regularity is assumed for the source. We first establish the strict convexity of $u$ in Theorem~\ref{strictness}. Let $v$ denote the potential associated with the optimal transport from $\Omega^*$ to $\Omega$. When $\dim\partial u(x)\le1$ and $p \in \ext\partial u(x),$ Theorem~\ref{estimate} proves that, for every $0<\varepsilon<1$,
\[
|Dv(y)-Dv(p)|\le C_{p,\varepsilon}|y-p|^{1-\varepsilon}
\qquad\text{for }|y-p|<r_{p,\varepsilon}.
\]
The constants may depend on the subdifferential point. By a countable decomposition, these H\"older estimates are sufficient to bound the Hausdorff dimension of the images of the skeletal faces.

Our argument is based on the monotonicity formula of Collins--Tong~\cite[Theorem~3.1]{CT} and the good--bad scale method of Chen--Li--Wang~\cite[Propositions~7.2--7.3]{CLiW}. We derive a localized identity for smooth dual pairs obtained by approximating the inverse potential $v$, and pass to the limit using uniform energy estimates. At an open target facet, we verify obliqueness and apply the local regularity argument of Chen--Liu~\cite[Section~5]{CL3}, together with the flat-target adaptation of Chen--Li--Liu~\cite[Appendix, Proposition~6.1]{CLL} and the section estimates and perturbation method of Chen--Liu--Wang~\cite{CLW3}.

The paper is organized as follows. In Section~\ref{preliminaries} we introduce notations used throughout the paper and some basic estimates for the solution. In Section~\ref{convexity} we prove strict convexity of $u$. In Section~\ref{radial} we prove the $C^{1,1-\varepsilon}$ estimate of $v$. Sections~\ref{singularities} and~\ref{freeboundaries} contain the proofs of Theorems~\ref{complete} and~\ref{partial}, respectively. The local regularity argument at a flat target facet is included in Appendix~\ref{flat}.

\section{Preliminaries}\label{preliminaries}
\subsection{Notations}
Suppose $w:\mathbb{R}^n\rightarrow \mathbb{R}$ is a convex function. We use $\partial w$ for the convex subdifferential, $Dw$ for the almost-everywhere gradient, and $w^*$ for the Legendre transform. The associated Monge-Amp\`ere measure is defined by
\[
\mu_w(E)=|\partial w(E)|
\]
for any measurable set $E\subset \mathbb{R}^n,$ where $|\cdot |$ denotes the $n$-dimensional Hausdorff measure. For orthogonal projections onto a hyperplane, $|\cdot|$ denotes the induced $(n-1)$-dimensional measure. For a convex set $G\subset\R^n$, we denote by $\ext G$ its set of extreme points and by $\relint G$ its interior relative to its affine hull. We write $tG:=\{ty:y\in G\}$ for dilation about the origin and denote by $z_G$ its barycenter when $G$ is bounded and $|G|>0$. For a set $E\subset\R^n$, we denote its convex hull by $\co E$. In the following, we write $C,c$ for positive constants, which may change from line to line, and indicate their dependencies when needed.

Let $\Omega^*$ be a bounded open polyhedral set. For each $p\in\overline{\Omega^*}$ there exist $r_p>0$ and finitely many closed full-dimensional convex polyhedral cones $\mathcal C_1,\ldots,\mathcal C_N$ with vertex at the origin such that up to a set of measure zero,
\begin{equation}\label{cones}
(\Omega^*-p)\cap B_{r_p}=\left(\bigcup_{j=1}^N\mathcal C_j\right)\cap B_{r_p}.
\end{equation}
Note that the cones may overlap.

\subsection{Basic estimates}
Let $\Omega\subset\R^n$ be a bounded open set with $|\Omega|=|\Omega^*|$, and let
\begin{equation}\label{source}
(Du)_\#\chi_\Omega=\chi_{\Omega^*}.
\end{equation}
We extend $u$ to $\R^n$ as follows
$$\tilde u(x):=\sup\{L(x): L\text{ is affine},\ L\le u\text{ in }\Omega,\ DL\in\Omega^*\}.$$
The extension agrees with $u$ in $\Omega$ and is globally Lipschitz. For simplicity of notation, we denote the extended function $\tilde{u}$ as $u.$ Let $v(y)=\sup_{x\in\Omega}\{x\cdot y-u(x)\}.$ It is also well known that $(Dv)_\#\chi_{\Omega^*}=\chi_{\Omega}.$ Moreover,
\begin{equation}\label{duality}
u\leq v^*\quad\hbox{in }\R^n,
\qquad u=v^*\quad\hbox{in }\Omega,
\end{equation}
and, for $x\in\Omega$, we have
\begin{equation}\label{contacts}
y\in\partial u(x)
\quad\Longleftrightarrow\quad x\in\partial v(y)
\quad\Longleftrightarrow\quad u(x)+v(y)=x\cdot y.
\end{equation}
These are the Legendre duality relations \cite[Appendix~A.4.2]{FiBook}, applied to $v^*$, which agrees with $u$ on the open set $\Omega$. On full-measure subsets of $\Omega$ and $\Omega^*$, the maps $Du$ and $Dv$ are inverse and preserve Lebesgue measure \cite{Br1}. Thus, for every Borel set $E\subset\R^n$,
\begin{equation}\label{measure}
|E\cap\Omega^*|\leq\mu_v(E),
\qquad |\partial v(E)\cap\Omega|\leq|E\cap\Omega^*|.
\end{equation}

Following \cite[Section~2]{CLW1}, we introduce two different sub-level sets of $u$ at $x_0\in \overline{\Omega}.$ One is the centred sub-level set
\[
S_h^c[u](x_0)=\{y\in \mathbb{R}^n:u(y)<u(x_0)+b\cdot(y-x_0)+h\},
\]
where $b$ is chosen so that $x_0$ is the mass center for $ S_h^c[u](x_0)$. When the graph of $u$ does not contain any straight line, existence and uniqueness of centered sections is well known, see for instance \cite[Theorem~A.7 and Lemma~A.8]{CM}. The other one is the sub-level set of $u$ with height $h$
\[
S_h[u](x_0)=\{y\in \Omega:u(y)<u(x_0)+p\cdot(y-x_0)+h\},
\qquad p\in\partial u(x_0).
\]
In the following we will write $S_h^c[u](x_0)$ and $S_h[u](x_0)$ as $S_h^c$ and $S_h$ when no confusion arises.

We recall an elementary centroid estimate for nonnegative affine functions. Let $G$ be a convex body with centroid $z_G$. Every nonnegative affine function $\ell$ on $G$ satisfies
\begin{equation}\label{affine}
\sup_G\ell\le(n+1)\ell(z_G).
\end{equation}
Indeed, if $M=\max_G\ell>0$ and $\ell(z)=M$, then $z+(1-t/M)(G-z)\subset\{\ell\ge t\}$ for $0<t<M$. Integrating in $t$ yields $\ell(z_G)=|G|^{-1}\int_G\ell\ge M/(n+1)$. In fact, by the same argument and Jensen's inequality, one can prove
\begin{equation}\label{concavity}
\sup_G\varphi\le(n+1)\varphi(z_G)
\end{equation}
for every nonnegative concave function $\varphi$ on $G$.

Let $S_h^c$ be a bounded centred section with centroid $z_h$. We write
\[
(tS_h^c)_{z_h}:=z_h+t(S_h^c-z_h),\qquad t>0.
\]
A dilation $tS_h$ without an external centre subscript is always taken about the origin. By \cite[Lemma~A.3]{CM}, there exists an invertible linear map $T_h$ such that
\begin{equation}\label{normalization}
B_1\subset T_h(S_h^c-z_h)\subset B_{n^{3/2}},\qquad
|\det T_h|\approx|S_h^c|^{-1}.
\end{equation}
Consequently,
\begin{equation}\label{balance}
-(S_h^c-z_h)\subset n^{3/2}(S_h^c-z_h).
\end{equation}
These estimates imply the following localization of centred sections of $v$.
\begin{lemma}\label{localization}
Suppose $v\ge0$ in $\R^n$ and $p\in\ext\{v=0\}$. If $z_k\to p$ in $\R^n$ and $h_k\to0^+$, then
\[
\diam S_{h_k}^c[v](z_k)\to0.
\]
\end{lemma}

\begin{proof}
Write $S_{h_k}^c:=S_{h_k}^c[v](z_k)=\{v<\ell_k\}$, where $\ell_k$ is affine and $\ell_k(z_k)=v(z_k)+h_k$. By \eqref{affine}, we have
\[
0\le\sup_{\overline{S_{h_k}^c}}v
\le(n+1)\bigl(v(z_k)+h_k\bigr)\to0.
\]
If $\operatorname{diam}S_{h_k}^c\not\to0$, by \eqref{balance}, along a subsequence, there exist $r>0$ and $e_k\in\mathbb S^{n-1}$ such that
\[
z_k\pm re_k\in\overline{S_{h_k}^c},\qquad e_k\to e\in\mathbb S^{n-1}.
\]
This implies $v(p\pm re)=0$, contradicting the fact that $p\in\ext\{v=0\}$.
\end{proof}

\begin{lemma}[\cite{FK1}]\label{depth}
Let $S_h^c=\{v<\ell_h\}$ and let $T_h$ satisfy \eqref{normalization}. Then
\begin{equation}\label{lowerdepth}
\left(\sup_{S_h^c}(\ell_h-v)\right)^n
\ge c(n)|(\tfrac12 S_h^c)_{z_h}\cap\Omega^*|\,|S_h^c|.
\end{equation}
If $p\in S_h^c$, $x\in\partial v(p)$, and $B_\sigma(x)\subset\Omega$ for some $\sigma>0$, then
\begin{equation}\label{upperdepth}
[\ell_h(p)-v(p)]^n
\le C\!\left(n,\frac{\diam\Omega}{\sigma}\right)
|S_h^c\cap\Omega^*|\,|S_h^c|\,
\operatorname{dist}\bigl(
T_h(p-z_h),\partial T_h(S_h^c-z_h)
\bigr).
\end{equation}
\end{lemma}
The lower bound is the Alexandrov estimate in \cite[Section~3.1, equation~(4)]{FK1}, combined with \eqref{measure} after normalization. The upper bound is \cite[Proposition~1]{FK1}, with source and target interchanged. Its proof applies using \eqref{measure}, $\partial v(\R^n)\subset\co\overline\Omega$, and $B_\sigma(x)\subset\Omega$.

The following consequence will be used in the proof of Theorem~\ref{strictness}.

\begin{corollary}\label{depthcor}
Let $S_h^c,p,x,\sigma$ be as in Lemma~\ref{depth}. If, for some $c_*>0$,
\[
\ell_h(p)-v(p)\ge c_*\sup_{S_h^c}(\ell_h-v),\qquad
|(\tfrac12 S_h^c)_{z_h}\cap\Omega^*|\ge c_*|S_h^c\cap\Omega^*|,
\]
then
\begin{equation}\label{contradiction}
\operatorname{dist}\bigl(T_h(p-z_h),\partial T_h(S_h^c-z_h)\bigr)
\ge c\!\left(n,\frac{\diam\Omega}{\sigma}\right)c_*^{n+1}>0.
\end{equation}
\end{corollary}

\begin{proof}
Since $p\in S_h^c$, \eqref{upperdepth} implies $|S_h^c\cap\Omega^*|>0$. By \eqref{lowerdepth} and the assumptions,
\[
[\ell_h(p)-v(p)]^n
\ge c_*^n\left(\sup_{S_h^c}(\ell_h-v)\right)^n
\ge c(n)c_*^{n+1}|S_h^c\cap\Omega^*|\,|S_h^c|.
\]
Combining this with \eqref{upperdepth}, we obtain \eqref{contradiction}.
\end{proof}

\section{Strict convexity }\label{convexity}

In this section, we prove strict convexity of $u$ for non-convex polyhedral targets and establish the local equation for $v$ needed in Section~\ref{radial}.

\begin{theorem}\label{strictness}
Assume \eqref{source}, and suppose that the representation \eqref{cones} holds at every $p\in\overline{\Omega^*}$. Then $u$ is strictly convex in $\Omega$. For $x\in\Omega$ and $p\in\partial u(x)$, there exists a neighborhood $\mathcal N_p$ of $p$ such that
\begin{equation}\label{localequation}
v\in C^1(\mathcal N_p),\qquad Dv(p)=x,\qquad
\mu_v\lfloor\mathcal N_p
=\chi_{\Omega^*}\,dy\lfloor\mathcal N_p.
\end{equation}

\end{theorem}
We first show that $\ext\partial u(x)\subset\overline{\Omega^*}$, so that the local cone representation \eqref{cones} applies at each extreme point.
\begin{lemma}\label{extremelocation}
If $x\in\Omega$, $\ext\partial u(x)\subset\overline{\Omega^*}$.
\end{lemma}
\begin{proof}
By \cite[Theorem~25.6]{Rockafellar}, every $p\in\ext\partial u(x)$ is a limit of gradients $Du(x_j)$ with $x_j\to x$, chosen from the full-measure set where $Du\in\Omega^*$. Thus $p\in\overline{\Omega^*}$.
\end{proof}

\begin{proof}[Proof of Theorem~\ref{strictness}]
We first prove
\begin{equation}\label{uniqueness}
x_0\in\Omega,\quad p\in\operatorname{ext}\partial u(x_0)
\quad\Longrightarrow\quad \partial v(p)=\{x_0\}.
\end{equation}
Suppose that $\partial v(p)$ is not a singleton, and choose $x_c\in\operatorname{relint}\partial v(p)\cap\Omega$. By \eqref{duality} and the supporting inequality at $x_0$, $u(x)=x\cdot p-v(p)$ for $x\in\partial v(p).$ Every supporting affine function at $x_c$ also contacts at $x_0$, so $\partial u(x_c)\subset\partial u(x_0)$ and $p$ remains extreme.

After translation, subtraction of affine functions, and an orthogonal change of coordinates, assume $x_c=p=0, v(0)=0,$ $0\in\operatorname{relint}\partial v(0)\cap\Omega,$ and $\operatorname{span}\partial v(0)=\R^{n_0}\times\{0\}, $ where $1\leq n_0\leq n$. Write $y=(y',y'')$ with $y'\in\R^{n_0}$. Then, there exists a constant $c>0$,
\begin{equation}\label{contactnormalization}
B_c^{n_0}\times\{0\}\subset\partial v(0),\qquad v(y)\geq c|y'|.
\end{equation}
Fix $B_\sigma(0)\subset\Omega$. Since $v(y)\geq\sigma|y|-C$, the set $\{v=0\}=\partial u(0)$ is compact and has $0$ as an extreme point.

\medskip
\noindent\textit{Step 1. Auxiliary sections.}
By Lemma~\ref{extremelocation}, $0\in\overline{\Omega^*}$. Use \eqref{cones} at $0$, and choose
\begin{equation}\label{directions}
\eta_j\in\operatorname{int}\mathcal C_j,\qquad
\eta_{j,1}\ne0,\qquad
a_j\in\partial v(0),\qquad
a_j\cdot\eta_j=\sup_{x\in\partial v(0)}x\cdot\eta_j.
\end{equation}
For $0<\varepsilon<1/4$, set $\delta=\varepsilon^2$ and write
\[
S_\varepsilon^c:=S_\varepsilon^c[v](0)=\{v-\beta_\varepsilon\cdot y<\varepsilon\}.
\]
Boundedness of the centered section and \eqref{concavity} imply $\beta_\varepsilon\in\operatorname{co}\overline{\Omega},$ and $v-\beta_\varepsilon\cdot y\geq-n\varepsilon\quad\text{in }\R^n.$ For $j=1,\ldots,N$, set
\[
\begin{gathered}
W_j(y):=v(y)-[(1-\varepsilon)a_j+\varepsilon\beta_\varepsilon]\cdot y,\\
Z_j:=\{y\in\R^n:W_j(y)<2\delta\},\qquad j=1,\ldots,N.
\end{gathered}
\]
The identity
\[
W_j=(1-\varepsilon)(v-a_j\cdot y)+\varepsilon(v-\beta_\varepsilon\cdot y)
\]
gives $W_j\geq-n\delta$ in $\R^n$ and, on $Z_j$,
\begin{equation}\label{thin}
0\leq v(y)-a_j\cdot y\leq2(n+2)\delta,\qquad
a_j\cdot y\geq-2(n+2)\delta.
\end{equation}
Moreover,
\[
Z_j\subset\{v-\beta_\varepsilon\cdot y<2\varepsilon\}
\subset(3S_\varepsilon^c)_0.
\]
Thus all $Z_j$ are bounded and shrink to $0$ by Lemma~\ref{localization}.

\medskip
\noindent\textit{Step 2. Barycenter estimate.}
Set $\mu(E)=|E\cap\bigcup_{j=1}^N\mathcal C_j|$. Choose $i$ maximizing $\mu(Z_i)$, and then choose $j$ maximizing $|Z_i\cap\mathcal C_j|$. Denote $K=Z_i\cap\mathcal C_j$, $q=\bary K.$ Then
\[
0<|K|,\qquad
\mu(Z_j)\leq\mu(Z_i)\leq N|K|.
\]
We claim that
\begin{equation}\label{drift}
\frac{a_i\cdot q}{\delta}\longrightarrow+\infty.
\end{equation}
Otherwise, after passing to a subsequence and fixing $i,j$, we have $a_i\cdot q\leq C_1\delta$ for some $C_1>0$. Apply \eqref{affine} to the nonnegative affine function $a_i\cdot y+2(n+2)\delta$ on $K$. This gives
\[
\sup_{K}\bigl(a_i\cdot y+2(n+2)\delta\bigr)
\leq(n+1)\bigl(C_1+2(n+2)\bigr)\delta.
\]
By \eqref{thin} and \eqref{contactnormalization}, it follows that $v(y)\leq C\delta,$ and $|y'|\leq C\delta$ for $y\in K.$ Since $a_i-a_j\in\R^{n_0}\times\{0\}$,
\[
W_j=W_i+(1-\varepsilon)(a_i-a_j)\cdot y
\leq C\delta\quad\text{on }K.
\]
Choose a fixed $\kappa\in(0,1)$ sufficiently small that $W_j\leq\delta/2$ on $\kappa K$.

Fix $d>2\kappa C/|\eta_{j,1}|$ and an integer $N_0>N\kappa^{-n}$. Upper semicontinuity of $\partial v$ at $0$, \eqref{directions}, and boundedness of $\beta_\varepsilon$ give $D_{\eta_j}^{+}W_j\leq o(1)$ along $\kappa K+[0,N_0d\delta]\eta_j$. Hence, for sufficiently small $\varepsilon$,
\[
\kappa K+kd\delta\eta_j\subset Z_j\cap\mathcal C_j,
\qquad k=0,\ldots,N_0-1.
\]
Since $d|\eta_{j,1}|\delta>2\kappa C\delta,$ these copies lie in disjoint intervals of the first coordinate, so
\[
N_0\kappa^n|K|
\leq\mu(Z_j)\leq\mu(Z_i)\leq N|K|,
\]
contrary to the choice of $N_0$. This proves \eqref{drift}.

\medskip
\noindent\textit{Step 3. Choosing section and contradiction.}
For $0<\lambda<1/4$, to be fixed below, set $h=(n+4)\delta$ and $z_h=\lambda q$, and take the centred section $S_h^c:=S_h^c[v](z_h)$, written as
\[
S_h^c=\{W_i<\ell\},\qquad \bary S_h^c=z_h,\qquad
\ell(z_h)=W_i(z_h)+h.
\]
Since $-n\delta\leq W_i(z_h)\leq2\lambda\delta$,
\[
4\delta\leq\ell(z_h)\leq(n+5)\delta,\qquad
\sup_{S_h^c}(\ell-W_i)\leq(n+1)(n+4)\delta.
\]
Applying \eqref{affine} to $\ell+n\delta$ on $S_h^c$ gives $\ell\leq C\delta$ there, uniformly in $\lambda$. The same bound holds on $Z_i\setminus S_h^c$, where $\ell\leq W_i<2\delta$. The identity $\ell(z_h)=(1-\lambda)\ell(0)+\lambda\ell(q)$ therefore gives, for a sufficiently small fixed $\lambda=\lambda(n)$, we have $2\delta\leq\ell(0)\leq C\delta$ and $\ell(q)\geq-C\delta.$ Hence
\begin{equation}\label{sectiondepth}
\ell(0)-W_i(0)\geq c_* \sup_{S_h^c}(\ell-W_i)
\end{equation}
for some constant $c_*$ independent of $\varepsilon.$ Applying \eqref{affine} to $C\delta-\ell$ on $K$ gives $\ell\geq-C\delta$ on $K$. Convexity then yields, for fixed dimensional constants $\vartheta>0$ small and $D>1$ large,
\[
\vartheta K\subset S_h^c\cap\mathcal C_j,\qquad S_h^c\subset DZ_i.
\]
Indeed, $W_i(\vartheta y)\leq2\vartheta\delta< \ell(\vartheta y)$ for $y\in K$, and $W_i(y/D)\leq W_i(y)/D<2\delta$ for $y\in S_h^c$.

Thus $S_h^c$ shrinks to $0$. Since $z_h\in\mathcal C_j$,
\begin{equation}\label{sectionmass}
\mu\bigl((\tfrac12 S_h^c)_{z_h}\bigr)
\geq2^{-n}|S_h^c\cap\mathcal C_j|
\geq2^{-n}\vartheta^n|K|
\geq c(n,N)\mu(S_h^c).
\end{equation}
For small $\varepsilon$, these cone masses coincide with target masses by \eqref{cones}.

By \eqref{thin} and $S_h^c\subset DZ_i$, $a_i\cdot y\geq-C\delta$ on $S_h^c$, while $a_i\cdot z_h/\delta\to+\infty$ by \eqref{drift}. Let $T_h$ satisfy \eqref{normalization}. Since $0\in S_h^c$, \eqref{normalization} gives $|T_h^{-T}a_i|\geq c(n)a_i\cdot z_h$. The containing halfspace therefore implies
\begin{equation}\label{sectiondistance}
\operatorname{dist}\bigl(-T_h z_h,
\partial T_h(S_h^c-z_h)\bigr)
\leq\frac{C\delta}{|T_h^{-T}a_i|}
\leq\frac{C\delta}{a_i\cdot z_h}\longrightarrow0.
\end{equation}
Since $0\in\partial v(0)$ and $B_\sigma(0)\subset\Omega$, Corollary~\ref{depthcor}, applied with $\ell_h(y)=\ell(y)+[(1-\varepsilon)a_i+\varepsilon\beta_\varepsilon]\cdot y$, yields \eqref{contradiction} at $p=0$ by \eqref{sectiondepth} and \eqref{sectionmass}, contradicting \eqref{sectiondistance}. Hence \eqref{uniqueness} holds.

\medskip
\noindent\textit{Step 4. Strict convexity and regularity of $v$.}
Suppose $u$ were affine on a nontrivial open segment $I\subset\Omega.$ Take any $x_0$ in the interior of $I.$ Let $p$ be an extreme point of $\partial u(x_0).$ Then $p\in \partial u(x)$ for any $x\in I,$ which implies that $I\subset \partial v(p).$ This contradicts \eqref{uniqueness}. Thus $u$ is locally strictly convex in $\Omega$.

Let $x\in\Omega$ and $p\in\partial u(x)$. If $z\in\partial v(p)$ and $z\ne x$, then \eqref{duality} implies that $u$ is affine on $[x,z]$, contradicting local strict convexity near $x$. Hence $\partial v(p)=\{x\}$. By upper semicontinuity, choose a neighborhood $\mathcal N_p$ with $\partial v(\mathcal N_p)\subset\Omega$. The same argument applies to every $y\in\mathcal N_p$, so $v\in C^1(\mathcal N_p)$ and $Dv(p)=x$. Equation \eqref{measure} now yields \eqref{localequation}.
\end{proof}

\section{A localized monotonicity formula}\label{radial}

In this section, we prove the $C^{1,1-\varepsilon}$ estimate of $v$ needed for the Hausdorff dimension bounds. We first derive a localized almost monotonicity formula by smooth approximation on the target side. Blow-up and iteration along a density-one set of scales then imply
\[
\operatorname{diam}\{u<h\}\le C_\eta h^{1/2-\eta},
\qquad 0<\eta<\tfrac12,
\]
for all sufficiently small $h>0$. This bounds the growth of $u$ from below, and Legendre duality then implies the following estimate for $v$.

\begin{theorem}\label{estimate}
Assume \eqref{source} and that $u$ is locally strictly convex in $\Omega$. Let $x_0\in\Omega$ satisfy $\dim\partial u(x_0)\le1$, and let $p_0\in\ext\partial u(x_0)$. If \eqref{cones} holds at $p_0$, then, for every $\varepsilon\in(0,1)$, there exist $C_\varepsilon,r_\varepsilon>0$ such that
\begin{equation}\label{gradient}
|Dv(y)-x_0|\le C_\varepsilon|y-p_0|^{1-\varepsilon}
\qquad y\in B_{r_\varepsilon}(p_0).
\end{equation}
\end{theorem}

\subsection{Monotonicity}
\label{secmon}

Our argument is based on the monotonicity formula of Collins--Tong~\cite[Theorem~3.1]{CT} and the good--bad scale method of Chen--Li--Wang~\cite[Propositions~7.2--7.3]{CLiW}. We approximate the inverse potential and introduce a cutoff near the chosen extreme point of the subdifferential to obtain a localized almost monotonicity formula.

After translation and subtraction of supporting affine functions, assume $x_0=p_0=0$ and $u(0)=v(0)=0$, so that $u,v\ge0$. Set $L=\partial u(0)$. Local strict convexity and Legendre duality give
\[
\{u=0\}=\{0\},\qquad
\{v=0\}=L,\qquad
\partial v(p)=\{0\}\quad(p\in L).
\]
Since $\dim L\le1$ and $0\in\ext L$, either $L=\{0\}$ or $L=[0,p_1]$ for some $p_1\ne0$. We first consider the latter case and set $\tau=p_1/|p_1|$.

By Theorem \ref{strictness}, there is a neighborhood $\mathcal N$ of $L$ such that $v\in C^1(\mathcal N)$, $Dv(\mathcal N)\Subset\Omega$, and $\mu_v=\chi_{\Omega^*}\,dy$ in $\mathcal N$. Choose a thin ellipsoid $D$ with $L\Subset D\Subset\mathcal N$. Note that $\min_{\partial D}v>0$. Choose $r^*>0$ so that $B_{4r^*}\Subset D$ and $\Omega^*\cap B_{4r^*}=\mathcal C\cap B_{4r^*}$ up to a set of measure zero, where $\mathcal C=\bigcup_j\mathcal C_j$ and the cones are those in \eqref{cones} at $0$.

Fix $0<b<\frac1{16}\min\{r^*,|p_1|\}$ and $\psi\in C^\infty(\mathbb R;[0,1])$ such that
\[
\psi=1\ \text{on }(-\infty,b],\qquad \psi=0\ \text{on }[2b,\infty),\qquad -C/b\le\psi'\le0,
\]
and let $\zeta(y)=\psi(\tau\cdot y)^2.$ Choose $r'>0$ sufficiently small that $B_{4r'}\Subset\Omega$ and $\partial u(B_{4r'})\subset D\cap(L+B_{b/8})$. Then $p\in\partial u(B_{4r'})$ and $\zeta(p)>0$ imply $|p|<r^*/2$.

Set $S_h=\{u<h\}.$ Let $-a_h\tau\in\partial S_h$ and set $\delta_h=h/a_h$. Since $\max_{p\in L}(-p\cdot\tau)=0$, we have $\delta_h\to0$ as $h \rightarrow 0$. Define $\Psi(x)=x\cdot Du(x)$, and
\[
M(h)=\int_{\{\Psi<h\}}\zeta(Du)\,dx,\qquad F(h)=h^{-n/2}M(h).
\]
If $L=\{0\}$, choose $\mathcal N,D,r^*$ by the same construction. Set $\zeta\equiv1$ and $\delta_h=0$, and choose $r'>0$ so that $B_{4r'}\Subset\Omega$ and $\partial u(B_{4r'})\subset D\cap B_{r^*/2}$.

We now establish an almost monotonicity formula for $F$. Since $u$ need not be differentiable, we derive the monotonicity formula by smooth approximation, avoiding a decomposition of the singular Hessian measure. For the smooth approximations $u_k$, define
\[
E_k(h)=\int_{\{x\cdot Du_k<h\}}\zeta(Du_k)\left|(D^2u_k)^{1/2}x-(D^2u_k)^{-1/2}Du_k\right|^2\,dx.
\]
The radial energy $E$ is defined as a locally uniform subsequential limit of $E_k$, whose existence is proved below.

\begin{proposition}\label{monotonicity}
For some $h_0>0$, the radial energy $E$ exists, is nonnegative and nondecreasing, and satisfies
\begin{equation}\label{radialbound}
0\le E(h)\le C_nhM(h).
\end{equation}
Moreover, $M,F,E$ are locally absolutely continuous on $(0,h_0)$, and there exists $R\in L^1_{\mathrm{loc}}(0,h_0)$ such that, for almost every $h\in(0,h_0)$,
\begin{equation}\label{identity}
-4h^{n/2+1}F'(h)=E'(h)-R(h),\qquad |R(h)|\le C\sqrt{\delta_h}M(h).
\end{equation}
\end{proposition}

\begin{proof}
\medskip
\noindent\textit{Step 1. Smooth approximation.}
Choose smooth densities $0<f_k\le1$ on $\overline D$ such that $f_k\to\chi_{\Omega^*}$ in $L^1(D)$ and, for some $\rho_k \rightarrow 0$,
\begin{equation}\label{radialdensity}
|y\cdot Df_k(y)|\le C1_{B_{2\rho_k}}(y)
\qquad(y\in B_{r^*}).
\end{equation}
Indeed, let $g_k$ be smooth positive approximations of $\chi_{\mathcal C\cap\mathbb S^{n-1}}$ on $\mathbb S^{n-1}$. For $2\rho_k\le |y|<r^*$, set $f_k(y)=g_k\left(\frac{y}{|y|}\right).$ Inside $B_{2\rho_k}$, modify $f_k$ by a radial cutoff so that it is constant near the origin and $|y\cdot Df_k|\le C$. Outside $B_{r^*}$, complete the construction smoothly so that $f_k\to\chi_{\Omega^*}$ in $L^1(D)$.

Approximate $v|_{\partial D}$ by smooth boundary data $\varphi_k$ and solve $\det D^2\hat v_k=f_k$ in $D$ by \cite[Theorem~1.1]{CNS}. By Caffarelli's stability estimate \cite[Lemma~1.3]{C96},
\[
\begin{aligned}
\|\hat v_k-v\|_{L^\infty(D)}
&\le \|\varphi_k-v\|_{L^\infty(\partial D)}
+C_D\|f_k^{1/n}-\chi_{\Omega^*}\|_{L^n(D)}\\
&\le \|\varphi_k-v\|_{L^\infty(\partial D)}
+C_D\|f_k-\chi_{\Omega^*}\|_{L^1(D)}^{1/n}
\longrightarrow0.
\end{aligned}
\]
Since $Dv(0)=0,$ by the above convergence and convexity we have $D\hat v_k(0)\to0$. Define
\[
v_k(y)=\hat v_k(y)-\hat v_k(0)-D\hat v_k(0)\cdot y,
\qquad
u_k(x)=\sup_{y\in\overline D}\{x\cdot y-v_k(y)\}.
\]
Then $v_k\ge0$, $v_k(0)=Dv_k(0)=0$, and $v_k\to v$ uniformly. Since $\min_{\partial D}v>0$, for all sufficiently small $|x|$ the maximizer defining $u_k(x)$ lies in $D$, uniformly in $k$. By taking $k$ large and decreasing $r'$, the functions $u_k$ are smooth on $B_{4r'}$ and satisfy
\[
\begin{gathered}
u_k(0)=Du_k(0)=0,\qquad \det D^2u_k=\frac1{f_k(Du_k)}\ge1,\\
u_k\to u\quad\text{locally uniformly in }B_{4r'}.
\end{gathered}
\]

Here the restricted dual of $v$ agrees with $u$ on this ball, since $\partial u(B_{4r'})\Subset D$. By convexity and local uniform convergence, for sufficiently small $r'$ and sufficiently large $k$, we have
\[
Du_k(B_{4r'})\subset
\begin{cases}
L+B_{b/4}, & L=[0,p_1],\\
B_{r^*/2}, & L=\{0\}.
\end{cases}
\]
Since $\zeta\equiv1$ in the singleton case, in both cases we obtain
\begin{equation}\label{approxcutoff}
x\in B_{4r'},\quad\zeta(Du_k(x))>0
\quad\Longrightarrow\quad |Du_k(x)|<r^*/2.
\end{equation}

\medskip
\noindent\textit{Step 2. Section and mass estimates.}
Let
\[
S_{k,h}=\{u_k<h\},\qquad
\Psi_k=x\cdot Du_k(x),\qquad
\Psi_k^*=y\cdot Dv_k(y),
\]
and define $M_k,F_k$ as above with $u_k$ in place of $u$. In the segment case let $-a_{k,h}\tau\in\partial S_{k,h}$ and set $\delta_{k,h}=h/a_{k,h}$. In the singleton case set $\delta_{k,h}=0$.

Since $\{u=0\}=\{0\}$, local uniform convergence and convexity imply that $\overline{S_{k,4h_0}}\subset B_{r'}$ for a common $h_0>0$. By convexity and the section volume estimate \cite[Lemma~2]{Caffarelli}, we have
\begin{equation}\label{halfsec}
\frac12S_{k,h}\subset\{\Psi_k<h\}\subset S_{k,h},
\qquad |S_{k,h}|\le C_nh^{n/2}.
\end{equation}

In the segment case, since $0\in S_{k,h}$ and $-a_{k,h}\tau\in\partial S_{k,h}$, convexity gives $\frac12S_{k,h}+[-a_{k,h}/2,0]\tau \subset S_{k,h}.$ Consequently, for every $x'\in\frac12\operatorname{proj}_{\tau^\perp}S_{k,h}$,
\[
\mathcal L^1\{t:x'+t\tau\in S_{k,h}\}
\ge\frac{a_{k,h}}2.
\]
Integrating over $x'$, we have
\[
|S_{k,h}|
\ge\frac{a_{k,h}}{2^n}
|\operatorname{proj}_{\tau^\perp}S_{k,h}|.
\]

Next, fix $x'\in\operatorname{proj}_{\tau^\perp}S_{k,h}$. Since $u_k\le\Psi_k$ and $\zeta(Du_k)\ne1$ implies $D_\tau u_k>b$, the conditions
\[
\Psi_k(x'+t\tau)<h,\qquad
\zeta(Du_k(x'+t\tau))\ne1
\]
imply
\[
u_k(x'+t\tau)<h,\qquad
D_\tau u_k(x'+t\tau)>b.
\]
By convexity, any two such values $t_1<t_2$ satisfy
\[
b(t_2-t_1)
\le u_k(x'+t_2\tau)-u_k(x'+t_1\tau)<h.
\]
Thus the set of such $t$ has measure at most $h/b$. Integrating over $x'$ and using the preceding estimate, we obtain
\begin{equation}\label{fibers}
\begin{gathered}
|\operatorname{proj}_{\tau^\perp}S_{k,h}|
\le\frac{2^n}{a_{k,h}}|S_{k,h}|,\\
|\{\Psi_k<h,\ \zeta(Du_k)\ne1\}|
\le\frac{2^n}{b}\delta_{k,h}|S_{k,h}|.
\end{gathered}
\end{equation}

Moreover, $\sup_k\delta_{k,h}\to0$ as $h \rightarrow 0$. Indeed, for each fixed $h\in(0,h_0)$, uniform convergence of $u_k$ implies $\delta_{k,h}\to\delta_h$. By convexity, $\delta_{k,h}$ is nondecreasing in $h$. Since $\delta_h\to0$ and $\delta_{k,h}\to0$ for each fixed $k$ by $Du_k(0)=0$, the conclusion follows. Decreasing $h_0$, by \eqref{halfsec} and \eqref{fibers} we have
\begin{equation}\label{approximatemass}
c_n|S_{k,h}|\le M_k(h)\le|S_{k,h}|\le C_nh^{n/2}.
\end{equation}
For $\dim\partial u(x_0)=0,$ this follows directly with $\zeta=1$.

At $y=Du_k(x)$, duality gives $Dv_k(y)=x$, $\Psi_k^*(y)=\Psi_k(x)$, and $dx=f_k(y)\,dy$. Consequently,
\[
M_k(h)=\int_D\zeta(y)1_{\{\Psi_k^*<h\}}f_k(y)\,dy.
\]
If $\zeta(y)>0$ and $\Psi_k^*(y)<h_0$, then $u_k(Dv_k(y))\le\Psi_k^*(y)<h_0$. Thus $Dv_k(y)\in B_{r'}$ and by \eqref{approxcutoff} we have $y$ in $B_{r^*/2}$. All weighted target fields below are therefore compactly supported in this ball.

For each $h\in(0,h_0)$, uniform convergence of $u_k$ implies that $S_{k,h}\to S_h$ in Hausdorff distance. By convexity, $\Psi(r_2\omega)\ge\frac{r_2}{r_1}\Psi(r_1\omega)$ for $r_2>r_1>0$ at differentiability points. Hence $|\{\Psi=h\}|=0$ for $h>0$. Since $Du_k\to Du$ almost everywhere, we obtain $M_k(h)\to M(h)$. In the segment case, section convergence also gives $a_{k,h}\to a_h$. By monotonicity in $h$ and continuity of $a_h$, this convergence is locally uniform, and so is $\delta_{k,h}\to\delta_h$. In the singleton case, $\delta_{k,h}=\delta_h=0$. Passing to the limit gives
\begin{equation}\label{mass}
\tfrac12S_h\subset\{\Psi<h\}\subset S_h,\qquad
c_n|S_h|\le M(h)\le|S_h|\le C_nh^{n/2}.
\end{equation}

\medskip
\noindent\textit{Step 3. The radial identity.}
Set $H_k=D^2u_k$, $G_k=D^2v_k$, and
\[
Q_k=|H_k^{1/2}x-H_k^{-1/2}Du_k|^2,
\qquad
E_k(h)=\int_{\{\Psi_k<h\}}\zeta(Du_k)Q_k\,dx.
\]
In particular, $E_k$ is nonnegative and nondecreasing. We shall first derive a differential identity for these smooth approximations.

By a direct computation, we have $D\Psi_k=Du_k+H_kx,$ and $D\Psi_k^*=Dv_k+G_ky.$ Since $H_k,G_k$ are positive definite, on the positive level sets we have
\[
x\cdot D\Psi_k=\Psi_k+x\cdot H_kx>0,
\qquad
y\cdot D\Psi_k^*=\Psi_k^*+y\cdot G_ky>0.
\]
Thus, for every $h\in(0,h_0)$, the level sets $\{\Psi_k=h\}$ and $\{\Psi_k^*=h\}$ are smooth hypersurfaces. The source sublevel sets lie in $B_{r'}$ by Step~2. Decreasing $h_0$ if necessary, uniform convergence and $\min_{\partial D}v>0$ ensure that $v_k>h_0$ on $\partial D$. Since $v_k\le\Psi_k^*$, the target sublevel sets are compactly contained in $D$ as well.

The coarea formula gives
\[
M_k'(h)=\int_{\{\Psi_k=h\}}\frac{\zeta(Du_k)}{|D\Psi_k|}\,d\HH^{n-1},
\]
with analogous formulas for other smooth integrands and on the target side. Apply the divergence theorem to $x\zeta(Du_k)$ on $\{\Psi_k<h\}$. Using the outward normal $D\Psi_k/|D\Psi_k|$ and the preceding coarea formula, we obtain
\[
\begin{aligned}
&nM_k(h)+\int_{\{\Psi_k<h\}}D\zeta(Du_k)\cdot H_kx\,dx\\
&\qquad=\frac{d}{dh}\int_{\{\Psi_k<h\}}
\zeta(Du_k)(\Psi_k+x\cdot H_kx)\,dx.
\end{aligned}
\]
Similarly, by applying the divergence theorem to $yf_k\zeta$ we have
\[
\begin{aligned}
&nM_k(h)+\int_{\{\Psi_k^*<h\}}
\bigl(f_ky\cdot D\zeta+\zeta y\cdot Df_k\bigr)\,dy\\
&\qquad=\frac{d}{dh}\int_{\{\Psi_k^*<h\}}
f_k\zeta(\Psi_k^*+y\cdot G_ky)\,dy.
\end{aligned}
\]
Recall that under $y=Du_k(x)$,
\[
\Psi_k^*(y)=\Psi_k(x),\qquad
G_k(y)=H_k(x)^{-1},\qquad f_k(y)\,dy=dx.
\]
To collect the terms where a derivative falls on the cutoff or the density, write
\[
\begin{aligned}
R_k(h)&=\int_{\{\Psi_k<h\}}
D\zeta(Du_k)\cdot(H_kx+Du_k)\,dx,\\
\Gamma_k(h)&=\int_D\zeta(y)1_{\{\Psi_k^*<h\}}
y\cdot Df_k(y)\,dy.
\end{aligned}
\]
Changing variables in the target integral and adding the two divergence identities, we find
\[
\begin{aligned}
2nM_k(h)+R_k(h)+\Gamma_k(h)
&=\frac{d}{dh}\int_{\{\Psi_k<h\}}
\zeta(Du_k)(Q_k+4\Psi_k)\,dx\\
&=E_k'(h)+4hM_k'(h).
\end{aligned}
\]
Here we used
\[
Q_k=x\cdot H_kx+Du_k\cdot H_k^{-1}Du_k-2\Psi_k,
\]
and, in the last equality, the fact that $\Psi_k=h$ on the level set in the coarea formula. Since
\[
4h^{n/2+1}F_k'(h)=4hM_k'(h)-2nM_k(h),
\]
we have proved
\begin{equation}\label{smoothwindow}
-4h^{n/2+1}F_k'(h)=E_k'(h)-R_k(h)-\Gamma_k(h).
\end{equation}
Finally, the support property in Step~2 and \eqref{radialdensity} imply $|\Gamma_k(h)|\le C\int_{B_{2\rho_k}}dy\le C\rho_k^n.$ Thus the error caused by differentiating $f_k$ tends to zero uniformly in $h$.

\medskip
\noindent\textit{Step 4. Energy and error estimates.}
We first bound the two quadratic terms in $Q_k$. The vector field $x(h-\Psi_k)$ vanishes on $\{\Psi_k=h\}$. Its divergence in $\{\Psi_k<h\}$ is
\[
\operatorname{div}\bigl(x(h-\Psi_k)\bigr)
=nh-(n+1)\Psi_k-x\cdot H_kx.
\]
Integrating, and using $\Psi_k\ge0$, gives
\[
\begin{aligned}
\int_{\{\Psi_k<h\}}x\cdot H_kx\,dx
&=nh|\{\Psi_k<h\}|-(n+1)\int_{\{\Psi_k<h\}}\Psi_k\,dx\\
&\le nh|S_{k,h}|.
\end{aligned}
\]
On the target side, integrate the divergence of $yf_k\zeta(h-\Psi_k^*)$ on $\{\Psi_k^*<h\}$. Again the boundary term vanishes, so
\[
\begin{aligned}
\int_{\{\Psi_k^*<h\}}f_k\zeta\,y\cdot G_ky\,dy
&=nhM_k(h)-(n+1)\int_{\{\Psi_k^*<h\}}f_k\zeta\Psi_k^*\,dy\\
&\quad+\int_{\{\Psi_k^*<h\}}
(\zeta y\cdot Df_k+f_ky\cdot D\zeta)(h-\Psi_k^*)\,dy.
\end{aligned}
\]
The second term on the right is nonpositive. In the segment case,
\[
y\cdot D\zeta(y)
=2(\tau\cdot y)\psi(\tau\cdot y)\psi'(\tau\cdot y)\le0,
\]
because $\psi'\le0$ and $\psi'$ is supported in $[b,2b]$; in the singleton case this term is zero. The remaining density term is bounded by $Ch\rho_k^n$, as above. Hence
\[
\int_{\{\Psi_k^*<h\}}f_k\zeta\,y\cdot G_ky\,dy
\le nhM_k(h)+Ch\rho_k^n.
\]
Using $0\le\zeta\le1$, changing variables in the second quadratic term, and discarding $-2\Psi_k\le0$, we conclude that
\[
\begin{aligned}
E_k(h)
&\le\int_{\{\Psi_k<h\}}x\cdot H_kx\,dx
+\int_{\{\Psi_k^*<h\}}f_k\zeta\,y\cdot G_ky\,dy\\
&\le nh|S_{k,h}|+nhM_k(h)+Ch\rho_k^n.
\end{aligned}
\]
Together with \eqref{approximatemass}, this proves
\begin{equation}\label{smoothenergy}
0\le E_k(h)\le C_nh\bigl(M_k(h)+\rho_k^n\bigr).
\end{equation}

We next estimate $R_k$. It vanishes in the singleton case, so assume that $L$ is a segment. The decomposition $H_kx+Du_k=(H_kx-Du_k)+2Du_k$ separates the part controlled by $Q_k$ from a term supported where the cutoff changes. Indeed, since $D\zeta(Du_k)=2\psi(D_\tau u_k)\psi'(D_\tau u_k)\tau$, we have
\[
\begin{aligned}
R_k(h)
&=2\int_{\{\Psi_k<h\}}\psi\psi'\,
\tau\cdot(H_kx-Du_k)\,dx\\
&\quad+4\int_{\{\Psi_k<h\}}\psi\psi' D_\tau u_k\,dx,
\end{aligned}
\]
where $\psi$ and $\psi'$ in this display are evaluated at $D_\tau u_k$. For the first integral, the pointwise inequality
\[
\begin{aligned}
|\tau\cdot(H_kx-Du_k)|
&=\bigl|H_k^{1/2}\tau\cdot
(H_k^{1/2}x-H_k^{-1/2}Du_k)\bigr|\\
&\le(\tau\cdot H_k\tau)^{1/2}Q_k^{1/2}
\end{aligned}
\]
and Cauchy--Schwarz give the bound
\[
2E_k(h)^{1/2}
\left(\int_{S_{k,h}}|\psi'(D_\tau u_k)|^2
\tau\cdot H_k\tau\,dx\right)^{1/2}.
\]
To estimate the above integral, fix $x'\in\operatorname{proj}_{\tau^\perp}S_{k,h}$ and restrict $u_k$ to the line parallel to $\tau$:
\[
I=\{t:x'+t\tau\in S_{k,h}\},\qquad
\varphi(t)=u_k(x'+t\tau).
\]
Here $I$ is an interval, $\varphi'=D_\tau u_k$, and $\varphi''=\tau\cdot H_k\tau>0$. The substitution $s=\varphi'(t)$ therefore yields
\[
\int_I|\psi'(\varphi'(t))|^2\varphi''(t)\,dt
=\int_{\varphi'(I)}|\psi'(s)|^2\,ds
\le\int_b^{2b}\frac{C}{b^2}\,ds\le\frac Cb.
\]
Integrating this estimate over $x'$ and using \eqref{fibers}, we obtain
\[
\int_{S_{k,h}}|\psi'(D_\tau u_k)|^2\tau\cdot H_k\tau\,dx
\le\frac Cb|\operatorname{proj}_{\tau^\perp}S_{k,h}|
\le\frac{C_n}{ba_{k,h}}|S_{k,h}|.
\]
For the second integral in $R_k$, notice that $\psi'\ne0$ only when $b<D_\tau u_k<2b$, so $|\psi\psi'D_\tau u_k|\le C$. This integrand vanishes wherever $\zeta(Du_k)=1$. The second estimate in \eqref{fibers} therefore gives
\[
4\int_{\{\Psi_k<h\}}|\psi\psi'D_\tau u_k|\,dx
\le C|\{\Psi_k<h,\ \zeta(Du_k)\ne1\}|
\le C\delta_{k,h}|S_{k,h}|.
\]
Consequently,
\[
|R_k(h)|\le C\left(\frac{E_k(h)|S_{k,h}|}{ba_{k,h}}\right)^{1/2}
+C\delta_{k,h}|S_{k,h}|.
\]
Using \eqref{approximatemass}, \eqref{smoothenergy}, and $h/a_{k,h}=\delta_{k,h}$, and taking $h_0$ smaller so that $\delta_{k,h}\le1$, we arrive at
\begin{equation}\label{smoothremainder}
|R_k(h)|\le C\sqrt{\delta_{k,h}}M_k(h)
+C\sqrt{\delta_{k,h}M_k(h)\rho_k^n}.
\end{equation}
Here $C$ may depend on the fixed cutoff parameter $b$, but not on $k$ or $h$.

\medskip
\noindent\textit{Step 5. Passage to the limit.}
By convexity, $\Psi_k(\lambda x)\le\lambda\Psi_k(x)$ for $0<\lambda<1$. Hence, for $0<s<t<h_0$,
\[
\frac{s}{t}\{\Psi_k<t\}\subset\{\Psi_k<s\},\qquad
0\le M_k(t)-M_k(s)
\le\bigl((t/s)^n-1\bigr)|\{\Psi_k<s\}|.
\]
Using \eqref{halfsec} and \eqref{approximatemass}, we obtain
\[
0\le M_k'(h)\le\frac nh|\{\Psi_k<h\}|
\le C_nh^{n/2-1},\qquad
|F_k'(h)|\le\frac{C_n}{h}.
\]
Thus the convergence $M_k\to M$ from Step~2 is locally uniform, and $M,F$ are locally Lipschitz. Moreover, \eqref{smoothwindow} gives
\[
0\le E_k'(h)
\le4h^{n/2+1}|F_k'(h)|+|R_k(h)|+|\Gamma_k(h)|,
\]
whose right-hand side is locally uniformly bounded by \eqref{smoothremainder} and $|\Gamma_k|\le C\rho_k^n$. Together with \eqref{smoothenergy}, this implies, along a subsequence,
\begin{equation}\label{energylimit}
E_k(h)\longrightarrow E(h)\qquad 0<h<h_0,
\end{equation}
locally uniformly in $h$, where $E$ is locally Lipschitz, nonnegative and nondecreasing, with $E(h)\le C_nhM(h)$.

Set $R=E'+4hM'-2nM$. Integrating \eqref{smoothwindow} by parts and passing to the limit, for $0<s<t<h_0$ we obtain
\[
\begin{aligned}
\int_s^tR_k(h)\,dh
&=E_k(t)-E_k(s)+4[hM_k(h)]_s^t
-(2n+4)\int_s^tM_k(h)\,dh
-\int_s^t\Gamma_k(h)\,dh\\
&\longrightarrow\int_s^tR(h)\,dh.
\end{aligned}
\]
By \eqref{smoothremainder} and the local uniform convergence of $M_k$ and $\delta_{k,h}$, we have
\[
\left|\int_s^tR(h)\,dh\right|
\le C\int_s^t\sqrt{\delta_h}M(h)\,dh.
\]
By Lebesgue differentiation, $|R(h)|\le C\sqrt{\delta_h}M(h)$ almost everywhere. The definition of $R$ then implies
\[
-4h^{n/2+1}F'(h)=2nM(h)-4hM'(h)=E'(h)-R(h),
\]
which proves \eqref{identity}.
\end{proof}


\subsection{Blow-up limit}

We normalize $S_h$ while keeping the base point at $0$. By \cite[Lemma~A.3]{CM}, choose a symmetric positive definite map $A_h$ such that $c_h=\bary(A_hS_h)$ satisfies
\begin{equation}\label{john}
c_h+B_1\subset A_hS_h\subset c_h+n^{3/2}B_1.
\end{equation}
Since $0\in S_h$, $|c_h|\le n^{3/2}$ and $A_hS_h\subset B_{2n^{3/2}}$. By the mass comparison and the normalization of $S_h$,
\begin{equation}\label{volume}
\det A_h^{-1}\approx_n |S_h|
\approx_n M(h)=h^{n/2}F(h),
\qquad
\|A_h^{-1}\|\le\operatorname{diam}S_h
\longrightarrow0.
\end{equation}
Let $\tilde u_h(x)=h^{-1}u(A_h^{-1}x)$. Then
\[
\{\tilde u_h<s\}=A_hS_{sh},\qquad
\partial u(A_h^{-1}x)=hA_h^T\partial\tilde u_h(x).
\]
In the segment case, $-a_h\tau\in\partial S_h$ implies
\begin{equation}\label{targetprojection}
|hA_h\tau|\le2n^{3/2}\delta_h\longrightarrow0.
\end{equation}

\begin{lemma}\label{sectionlimits}
Suppose $h_j\to0$ and, for every $s>0$,
\begin{equation}\label{blowupcriterion}
\frac{M(sh_j)}{M(h_j)}\longrightarrow s^{n/2},
\qquad
\frac{E(sh_j)}{h_jM(h_j)}\longrightarrow0.
\end{equation}
Then, after passing to a subsequence, there is a compact convex set $K$ with nonempty interior such that
\[
\overline{A_{h_j}S_{sh_j}}\longrightarrow\sqrt{s}\,K
\quad\text{in Hausdorff distance for every }s>0.
\]
\end{lemma}

\begin{proof}
\medskip
\noindent\textit{Step 1. Compactness.}
Since $u\ge0$ and $u(0)=0$, convexity gives
\[
\min\{1,s\}S_h\subset S_{sh}
\subset\max\{1,s\}S_h
\]
and
\[
S_{sh}\subset S_{th}\subset(t/s)S_{sh}
\qquad(0<s<t).
\]
For each fixed $s>0$, \eqref{john} therefore bounds $A_{h_j}S_{sh_j}$ in a fixed ball and places a ball of radius $\min\{1,s\}$ inside it. By Blaschke's selection theorem and a diagonal argument, a subsequence converges at every positive rational height.

The second inclusion also gives
\[
d_{\mathrm H}\!\left(
\overline{A_{h_j}S_{sh_j}},
\overline{A_{h_j}S_{th_j}}
\right)
\le C_n\max\{1,s\}\frac{t-s}{s}
\qquad(0<s<t).
\]
Thus the same subsequence converges at every positive height, locally uniformly in the height:
\begin{equation}\label{sectionlimit}
\overline{A_{h_j}S_{sh_j}}\longrightarrow\overline{K_s}.
\end{equation}
Here $K_s$ denotes the interior of the limiting convex body. Each $K_s$ is bounded, nonempty and convex, and $0\in\overline{K_s}$.

Let $\Omega_\infty=\bigcup_{s>0}K_s.$ The inclusions above pass to the limit, so the sets $K_s$ are nested and $\Omega_\infty$ is open and convex. If $G\Subset K_s$, convexity and \eqref{sectionlimit} imply $G\subset A_{h_j}S_{sh_j}$ for all sufficiently large $j$. Hence $0\le\tilde u_{h_j}<s$ on $G$. By convexity, after a further subsequence, $\tilde u_{h_j}\longrightarrow\tilde u_\infty \ \text{locally uniformly in }\Omega_\infty, $ where $\tilde u_\infty$ is finite, nonnegative and convex.

Section convergence also implies $\tilde u_{h_j}(x)\to+\infty$ for $x\notin\overline{\Omega_\infty}$. Since $rK_s\subset K_{rs}$ for $0<r<1$, we have $r\Omega_\infty\subset\Omega_\infty$ and, by convexity,
\begin{equation}\label{radialconvexity}
0\le\tilde u_\infty(rx)\le r\tilde u_\infty(x)
\qquad(x\in\Omega_\infty,\ 0<r<1).
\end{equation}

\medskip
\noindent\textit{Step 2. Limit of the energy estimate.}
Let $J_j=\det A_{h_j}^{-1}\approx_n M(h_j)$ by \eqref{volume}. For the approximations in Proposition~\ref{monotonicity}, local uniform convergence and \eqref{energylimit} allow us to choose an increasing sequence $k(j)$ such that
\[
\|u_{k(j)}-u\|_{L^\infty(B_{r'})}\le h_j/j
\]
and
\[
E_{k(j)}(mh_j)\le E(mh_j)+h_jJ_j/j
\]
for every integer $1\le m\le j$ with $mh_j<h_0$.

Set $w_j(x)=h_j^{-1}u_{k(j)}(A_{h_j}^{-1}x)$. By \eqref{volume} and the choice of $k(j)$, $w_j\to\tilde u_\infty$ locally uniformly in $\Omega_\infty$. Fix bounded open sets $G\Subset G'\Subset\Omega_\infty$. Convexity implies $\sup_{G'}|Dw_j|\le C_{G'}$. Choosing $\chi\in C_c^\infty(G')$ with $\chi\ge0$ and $\chi=1$ on $G$, we obtain
\begin{equation}\label{traceb}
\int_G\Delta w_j\,dx
\le\int\chi\Delta w_j\,dx
=\int w_j\Delta\chi\,dx
\le C_G.
\end{equation}

Since $Du_{k(j)}(A_{h_j}^{-1}x)=h_jA_{h_j}^TDw_j(x)$, in the segment case \eqref{targetprojection} implies
\[
\sup_G
\left|\tau\cdot h_jA_{h_j}^TDw_j\right|
\le |h_jA_{h_j}\tau|\sup_G|Dw_j|
\longrightarrow0.
\]
Hence $\zeta(h_jA_{h_j}^TDw_j)=1$ on $G$ for all large $j$. In the singleton case $\zeta\equiv1$.

Choose a fixed integer $m$ such that $x\cdot Dw_j(x)<m$ on $G$ for all large $j$. Under the change of variables, we have
\[
\begin{gathered}
D^2u_{k(j)}(A_{h_j}^{-1}x)
=h_jA_{h_j}^TD^2w_j(x)A_{h_j},\\
\Psi_{k(j)}(A_{h_j}^{-1}x)
=h_jx\cdot Dw_j(x).
\end{gathered}
\]
Expanding the square in the definition of $Q_{k(j)}$ therefore gives
\[
\begin{aligned}
&\int_G
\left|(D^2w_j)^{1/2}x
-(D^2w_j)^{-1/2}Dw_j\right|^2\,dx\\
&\qquad=
\frac{1}{h_jJ_j}
\int_{A_{h_j}^{-1}G}Q_{k(j)}\,dx
\le\frac{E_{k(j)}(mh_j)}{h_jJ_j}
\le\frac{E(mh_j)}{h_jJ_j}+\frac1j
\longrightarrow0
\end{aligned}
\]
by \eqref{blowupcriterion} and $J_j\approx_n M(h_j)$.

Since $D(x\cdot Dw_j-2w_j)=D^2w_jx-Dw_j$ and
\begin{equation}\label{tracebound}
|D^2w_jx-Dw_j|^2
\le
\tr D^2w_j\,
\left|(D^2w_j)^{1/2}x
-(D^2w_j)^{-1/2}Dw_j\right|^2,
\end{equation}
\eqref{traceb} and Cauchy--Schwarz yield
\begin{equation}\label{euler}
\left(
\int_G|D(x\cdot Dw_j-2w_j)|\,dx
\right)^2
\le C_G\frac{E_{k(j)}(mh_j)}{h_jJ_j}
\longrightarrow0.
\end{equation}

\medskip
\noindent\textit{Step 3. Radial homogeneity.}
Local uniform convergence and convexity imply
\[
x\cdot Dw_j-2w_j
\longrightarrow
x\cdot D\tilde u_\infty-2\tilde u_\infty
\quad\text{in }L^1_{\mathrm{loc}}(\Omega_\infty).
\]
By \eqref{euler} and connectedness of $\Omega_\infty$, $x\cdot D\tilde u_\infty-2\tilde u_\infty=c_0$ almost everywhere, for some $c_0\in\R$.

The function $\tilde u_\infty$ is locally Lipschitz. Integrating along almost every ray gives
\[
\tilde u_\infty(rx)=r^2\tilde u_\infty(x)+\frac{c_0}{2}(r^2-1).
\]
By \eqref{radialconvexity}, letting $r\to0$ yields $c_0=0$. Continuity therefore gives
\begin{equation}\label{radialscaling}
\tilde u_\infty(rx)=r^2\tilde u_\infty(x)
\qquad(x\in\Omega_\infty,\ 0<r<1).
\end{equation}

By the convergence of sections and \eqref{radialscaling}, we have $K_s=\{x\in\Omega_\infty:\tilde u_\infty(x)<s\}.$ Moreover, each ray meets the level set $\{\tilde u_\infty=s\}$ at most once. Hence $|\{\tilde u_\infty=s\}|=0$ for $s>0.$

\medskip
\noindent\textit{Step 4. Scaling of the limiting sections.}
Fix $s>0$. Changing variables in the definition of $M$, we obtain
\[
\frac{M(2sh_j)}{J_j}
=
\int_{\mathbb R^n}
\zeta\!\left(h_jA_{h_j}^TD\tilde u_{h_j}(x)\right)
\mathbf1_{\{x\cdot D\tilde u_{h_j}(x)<2s\}}\,dx.
\]
The integrands are bounded by one and supported in $A_{h_j}S_{2sh_j}$, hence in a fixed ball. Almost-everywhere gradient convergence, \eqref{targetprojection} and \eqref{radialscaling} imply
\[
x\cdot D\tilde u_{h_j}\longrightarrow2\tilde u_\infty,
\qquad
\zeta\!\left(h_jA_{h_j}^TD\tilde u_{h_j}\right)\longrightarrow1
\quad\text{almost everywhere in }\Omega_\infty.
\]
Outside $\overline{\Omega_\infty}$, $x\cdot D\tilde u_{h_j}\ge\tilde u_{h_j}\to+\infty$. Since $|\partial\Omega_\infty|=|\{\tilde u_\infty=s\}|=0$, dominated convergence yields
\[
\frac{M(2sh_j)}{J_j}\longrightarrow|K_s|.
\]

Since $0<|K_1|<\infty$, \eqref{blowupcriterion} now implies
\[
\frac{|K_s|}{|K_1|}
=
\lim_{j\to\infty}
\frac{M(2sh_j)/M(h_j)}{M(2h_j)/M(h_j)}
=
\frac{(2s)^{n/2}}{2^{n/2}}
=s^{n/2}.
\]
On the other hand, \eqref{radialscaling} gives $rK_s\subset K_{r^2s}$ for $0<r<1$. It follows that
$$|rK_s|=r^n|K_s|=|K_{r^2s}|.$$
Since both sets are open and convex, it follows that $rK_s=K_{r^2s}.$ Therefore, $K_s=\sqrt{s}\,K_1$ for $s>0$. Taking $K=\overline{K_1}$ and using \eqref{sectionlimit} completes the proof.
\end{proof}



\subsection{Good and bad scales}

\begin{proof}[Proof of Theorem~\ref{estimate}]
\medskip
\noindent\textit{Step 1. Variation of $\log F$.}
Fix $\theta\in(0,1)$ and set
\[
h_j=h_0\theta^j,\qquad
t_j=-\log h_j,\qquad
\ell(t)=\log F(e^{-t}).
\]
By \eqref{mass}, $\ell$ is bounded above. Since $E'\ge0$, \eqref{identity} gives
\[
\ell'(t)
=\frac{E'(e^{-t})-R(e^{-t})}{4M(e^{-t})}
\ge-C\sqrt{\delta_{e^{-t}}}.
\]
Writing $t_0=-\log h_0$, we obtain
\[
\begin{aligned}
\int_{t_0}^T|\ell'(t)|\,dt
&=\ell(T)-\ell(t_0)
+2\int_{t_0}^T(-\ell'(t))_+\,dt\\
&\le C+C\int_{t_0}^T\sqrt{\delta_{e^{-t}}}\,dt
=o(T),
\end{aligned}
\]
because $\delta_h\to0$ as $h \rightarrow 0$.

We extend $\ell'$ by zero below $t_0$. Fix an integer $m\ge1$. Since $t_{j+1}-t_j=|\log\theta|$, the intervals $[t_j-m,t_j+m]$ have bounded overlap. Hence
\[
\begin{aligned}
&\#\left\{1\le j\le J:
\int_{t_j-m}^{t_j+m}|\ell'(t)|\,dt>\frac1m
\right\}\\
&\qquad\le
C_{\theta,m}
\int_{t_0}^{t_J+m}|\ell'(t)|\,dt=o(J).
\end{aligned}
\]
Thus, for each fixed $m$, the number of indices where the variation exceeds $1/m$ is $o(J)$ as $J\to\infty$.

\medskip
\noindent\textit{Step 2. Mass ratios and vanishing energy.}
Suppose $j\to\infty$ along a sequence satisfying
\[
\int_{t_j-m}^{t_j+m}|\ell'(t)|\,dt\longrightarrow0
\qquad\text{for every fixed }m\ge1.
\]
We show that \eqref{blowupcriterion} holds along this sequence.

For any fixed $s>0$, choose $m>|\log s|$. Then
\[
\left|\log\frac{F(sh_j)}{F(h_j)}\right|
\le\int_{t_j-m}^{t_j+m}|\ell'(t)|\,dt
\longrightarrow0.
\]
The convergence is uniform when $s$ ranges over a compact subset of $(0,\infty)$. Consequently,
\begin{equation}\label{massratio}
\frac{F(sh_j)}{F(h_j)}\longrightarrow1,
\qquad
\frac{M(sh_j)}{M(h_j)}\longrightarrow s^{n/2}.
\end{equation}

Next, fix $s>0$ and $0<\sigma<1$. Integrating \eqref{identity} gives
\[
\begin{aligned}
\frac1{F(h_j)}
\int_{\sigma sh_j}^{sh_j}
\frac{E'(r)}{r^{n/2+1}}\,dr
&=4\frac{F(\sigma sh_j)-F(sh_j)}{F(h_j)}\\
&\quad+\frac1{F(h_j)}
\int_{\sigma sh_j}^{sh_j}
\frac{R(r)}{r^{n/2+1}}\,dr.
\end{aligned}
\]
The first term tends to zero by \eqref{massratio}. Since $|R(r)|\le C\sqrt{\delta_r}\,r^{n/2}F(r)$, the second term is bounded in absolute value by
\[
C\sup_{0<r\le sh_j}\sqrt{\delta_r}
\int_{\sigma s}^s
\frac{F(th_j)}{F(h_j)}\,\frac{dt}{t},
\]
which also tends to zero.

Since $E$ is nondecreasing,
\[
E(sh_j)\le E(\sigma sh_j)
+(sh_j)^{n/2+1}
\int_{\sigma sh_j}^{sh_j}
\frac{E'(r)}{r^{n/2+1}}\,dr.
\]
Using \eqref{radialbound} and \eqref{massratio}, we obtain
\[
\limsup_{j\to\infty}
\frac{E(sh_j)}{h_jM(h_j)}
\le C_n(\sigma s)^{n/2+1}.
\]
Letting $\sigma \rightarrow 0$ gives
\[
\frac{E(sh_j)}{h_jM(h_j)}\longrightarrow0.
\]
Thus both conditions in \eqref{blowupcriterion} hold along the chosen sequence.

\medskip
\noindent\textit{Step 3. Contraction at good scales.}
Fix $\eta\in(0,1/2)$. We claim that there exist an integer $m_0\ge1$ and $J_0$ such that
\[
\int_{t_j-m_0}^{t_j+m_0}|\ell'(t)|\,dt
\le\frac1{m_0},\qquad j\ge J_0,
\]
implies
\begin{equation}\label{goodcontraction}
\operatorname{diam}S_{h_{j+1}}
\le\theta^{1/2-\eta/2}
\operatorname{diam}S_{h_j}.
\end{equation}

Suppose otherwise. For each integer $m\ge1$, we can choose an increasing sequence of indices $j_m\to\infty$ such that
\[
\int_{t_{j_m}-m}^{t_{j_m}+m}|\ell'(t)|\,dt
\le\frac1m,
\]
but \eqref{goodcontraction} fails at $j_m$.

For every fixed integer $k_0\ge1$ and $m\ge k_0$,
\[
\int_{t_{j_m}-k_0}^{t_{j_m}+k_0}|\ell'(t)|\,dt
\le\frac1m\longrightarrow0.
\]
By Step 2, \eqref{blowupcriterion} holds along $h_{j_m}$. Lemma~\ref{sectionlimits} therefore gives, after passing to a subsequence,
\[
\overline{A_{h_{j_m}}S_{h_{j_m}}}\longrightarrow K,
\qquad
\overline{A_{h_{j_m}}S_{\theta h_{j_m}}}
\longrightarrow\sqrt\theta\,K.
\]

Thus, for some $\varepsilon_m\to0$,
\[
\overline{A_{h_{j_m}}S_{\theta h_{j_m}}}
\subset
\sqrt\theta\,\overline{A_{h_{j_m}}S_{h_{j_m}}}
+\varepsilon_m\overline{B_1}.
\]
Applying $A_{h_{j_m}}^{-1}$ and using \eqref{volume}, we obtain
\[
\begin{aligned}
\operatorname{diam}S_{\theta h_{j_m}}
&\le\sqrt\theta\,\operatorname{diam}S_{h_{j_m}}
+2\varepsilon_m\|A_{h_{j_m}}^{-1}\|\\
&\le(\sqrt\theta+2\varepsilon_m)
\operatorname{diam}S_{h_{j_m}}.
\end{aligned}
\]
Since $\sqrt\theta<\theta^{1/2-\eta/2}$, this contradicts the failure of \eqref{goodcontraction} for all sufficiently large $m$.

\medskip
\noindent\textit{Step 4. Iteration over all scales.}
Fix $m_0$ and $J_0$ as in Step 3, and define
\[
\mathcal I=
\left\{j\ge J_0:
\int_{t_j-m_0}^{t_j+m_0}|\ell'(t)|\,dt
\le\frac1{m_0}\right\}.
\]
Since $m_0$ is now fixed, Step 1 gives $\#\bigl(\{1,\ldots,J\}\setminus\mathcal I\bigr) =o(J).$ Thus \eqref{goodcontraction} holds at all but $o(J)$ indices up to $J$.

At the remaining indices, nesting gives $\operatorname{diam}S_{h_{j+1}} \le\operatorname{diam}S_{h_j}.$ Iteration therefore yields $\operatorname{diam}S_{h_j} \le C_\eta \theta^{(1/2-\eta/2)(j-o(j))}.$ For all sufficiently large $j$,
$$(1/2-\eta/2)(j-o(j))
\ge(1/2-\eta)j.$$

Hence, enlarging $C_\eta$ if necessary, $ \operatorname{diam}S_{h_j} \le C_\eta h_j^{1/2-\eta}.$ For $h_{j+1}<h\le h_j$, use $S_h\subset S_{h_j}$ and $h_j<h/\theta$. Absorbing the fixed factor into $C_\eta$, we obtain
\begin{equation}\label{diameter}
\operatorname{diam}S_h\le C_\eta h^{1/2-\eta}
\qquad(0<h\le h_0).
\end{equation}

\medskip
\noindent\textit{Step 5. Legendre duality.}
For $y$ sufficiently close to $0$, set $x=Dv(y)$. If $x=0$, the desired estimate is immediate. Otherwise take $h=2u(x)$ in \eqref{diameter}. Since $0,x\in S_h$, $|x|\le C_\eta u(x)^{1/2-\eta},$ and hence $c_\eta|x|^{2/(1-2\eta)}\le u(x).$ Since $v\ge0$, Legendre duality gives $u(x)=x\cdot y-v(y)\le|x||y|.$ Consequently, $|Dv(y)|\le C_\eta|y|^{(1-2\eta)/(1+2\eta)}.$ Choosing $\eta=\varepsilon/(4-2\varepsilon)$ and changing back to the original coordinates proves \eqref{gradient}.
\end{proof}

\section{Singularities in optimal transportation}\label{singularities}
We prove Theorem~\ref{complete} using the $C^{1,1-\varepsilon}$ estimate established in the last section. Under its hypotheses, we recall that
\[
A=\{x\in\Omega:\partial u(x)\text{ is not a singleton}\},
\qquad\Sigma=\overline A\cap\Omega.
\]
Lemma~\ref{extremelocation} implies
\begin{equation}\label{extremes}
\ext\partial u(x)\subset\overline{\Omega^*}
\qquad(x\in\Omega).
\end{equation}

By \cite[Theorem~4.1]{AAC}, if $\phi$ is finite and convex on an open convex set $G\subset\R^n$, then
\[
\dim_{\HH}\{x\in G:\dim\partial \phi(x)\ge k\}\le n-k,
\qquad 1\le k\le n.
\]

To control the images of target faces, we use the following consequence of pointwise H\"older estimates.

\begin{lemma}\label{dimension}
Let $P$ be a subset of an affine $k$-plane and let $T:P\to\R^n$ be bounded. Suppose that, for each $p\in P$ and $\alpha\in(0,1)$, there are $c_{p,\alpha},r_{p,\alpha}>0$ such that
\[
|T(q)-T(p)|\le c_{p,\alpha}|q-p|^\alpha
\qquad(q\in P\cap B_{r_{p,\alpha}}(p)).
\]
Then $\dim_{\HH}T(P)\le k$.
\end{lemma}
\begin{proof}
Fix $\alpha\in(0,1)$ and write $P=\bigcup_{j\ge1}P_j$, where $P_j=\{p\in P:c_{p,\alpha}\le j,\ r_{p,\alpha}\ge j^{-1}\}.$ For any $p, q\in P_j,$ if $|p-q|<j^{-1},$ by the assumption we have $|T(q)-T(p)|\le j|q-p|^\alpha.$ If $|p-q|\geq j^{-1}$, boundedness of $T$ implies $|T(q)-T(p)|\le2\bigl(\sup_P|T|\bigr)j^\alpha|q-p|^\alpha.$ Thus $T|_{P_j}$ is $\alpha$-H\"older and $\dim_{\HH}T(P_j)\le k/\alpha$. Taking the countable union and letting $\alpha\uparrow1$ proves the assertion.
\end{proof}

For a segment subdifferential, the regularity estimate for $Dv$ excludes intersections of its relative interior with $\overline{\Omega^*}$.

\begin{lemma}\label{middle}
Let $x_0\in\Omega$ and $\partial u(x_0)=[p_1,p_2]$, where $p_1\ne p_2$. Assume \eqref{gradient} at both extreme points for one fixed $\varepsilon\in(0,1/n)$. Then
\[
[p_1,p_2]\cap\overline{\Omega^*}=\{p_1,p_2\}.
\]
\end{lemma}
\begin{proof}
Both extreme points belong to $\overline{\Omega^*}$ by \eqref{extremes}. Let
\[
\tilde{v}(y)=v(y)-x_0\cdot y+u(x_0),\qquad
\tau=(p_2-p_1)/|p_2-p_1|.
\]
Then $\tilde{v}\ge0$ and $\tilde{v}=0$ on $[p_1,p_2]$. By integrating the $C^{1,1-\varepsilon}$ estimates at $p_1$ and $p_2$, we obtain $r_0>0$ and $C>0$ such that
\begin{equation}\label{growth}
0\le\tilde{v}(p_i+Y)\le C|Y|^{2-\varepsilon}
\qquad(|Y|<r_0,\ i=1,2).
\end{equation}
Suppose that $p_3\in(p_1,p_2)\cap\overline{\Omega^*}$. Choose $\ell>0$ with $p_3+[-2\ell,2\ell]\tau\subset(p_1,p_2)$. By convexity and \eqref{growth} we have
\begin{equation}\label{tube}
0\le\tilde{v}(p_3+s\tau+Y)\le C|Y|^{2-\varepsilon}
\qquad(|s|\le2\ell,\ Y\perp\tau,\ |Y|<r_0).
\end{equation}
For $0<\rho<\min\{r_0/4,\ell\}$, set
\[
\mathcal T_\rho=\{p_3+s\tau+Y:|s|<\ell,\ Y\perp\tau,\ |Y|<\rho\}.
\]
By convexity and \eqref{tube} again, for $y\in\mathcal T_{\rho/2}$ and $\xi\in\partial \tilde{v}(y)$,
\[
|\xi\cdot\tau|\le C\ell^{-1}\rho^{2-\varepsilon},\qquad
|\xi-(\xi\cdot\tau)\tau|\le C\rho^{1-\varepsilon}.
\]
It follows that
\begin{equation}\label{tubemass}
\mu_v(\mathcal T_{\rho/2})\le C\rho^{n+1-n\varepsilon}.
\end{equation}
By \eqref{cones} at $p_3$, $|\Omega^*\cap B_t(p_3)|\ge c_{p_3}t^n$ for small $t>0$. Since $B_{\rho/4}(p_3)\subset\mathcal T_{\rho/2}$, \eqref{measure} and \eqref{tubemass} imply $c\rho^n\le\mu_v(\mathcal T_{\rho/2}) \le C\rho^{n+1-n\varepsilon}.$ This is impossible as $\rho \rightarrow 0$, because $n\varepsilon<1$.
\end{proof}

We next prove local regularity for segment subdifferential sets on two open target facets.

\begin{proposition}\label{regularity}
Suppose $x_0\in A$, $\partial u(x_0)=[p_1,p_2]$, and both $p_i$ lie in open boundary facets of $\Omega^*$. There is $r>0$ such that
\[
\Sigma\cap B_r(x_0)=A\cap B_r(x_0)
\]
is a smooth hypersurface.
\end{proposition}
\begin{proof}
Choose disjoint closed half-balls $V_i=\overline{\Omega^*\cap B_\rho(p_i)}$ and set $u_i(x)=\sup_{y\in V_i}\{x\cdot y-v(y)\}$ for $i=1,2.$ Theorem~\ref{strictness} and Lemma~\ref{localization} provide the local equation $\mu_v=\chi_{V_i}\,dy$ and uniform localization of centered sections near $p_i$. The affine doubling estimate \cite[Lemma~7.5]{CM} and the quantitative convexity argument \cite[Theorem~7.13, proof of (7.15)--(7.16)]{CM}, followed by \cite[Remark~7.10]{CM}, yield $u_i\in C^{1,\beta}$ near $x_0$. Lemma~\ref{middle} and the local separation argument of \cite[Proposition~3.1]{CLL} give
\[
u=\max\{u_1,u_2\},\qquad A=\{u_1=u_2\},\qquad \partial u(x)=[Du_1(x),Du_2(x)]\quad(x\in A)
\]
in a smaller ball, with $Du_i(A)$ contained in the corresponding open facet. Thus $A$ is a relatively closed $C^{1,\beta}$ hypersurface there. If $\nu_i$ and $\nu_i^*$ are the inner source and target normals, Lemma~\ref{middle} implies
\[
\nu_1(x_0)\cdot\nu_1^*=\frac{(p_1-p_2)\cdot\nu_1^*}{|p_1-p_2|}>0,\qquad \nu_2(x_0)\cdot\nu_2^*=\frac{(p_2-p_1)\cdot\nu_2^*}{|p_1-p_2|}>0.
\]
Proposition~\ref{boundary}(i), the formula $\nu_i=(Du_i-Du_{3-i})/|Du_i-Du_{3-i}|$, and part~(ii) yield $u_i\in C^{2,\alpha}$ up to $A$. The boundary Schauder bootstrap in \cite[Remark~4.4]{CLW3} then proves smoothness. The local description also gives $\Sigma=A$.
\end{proof}

It remains to characterize the points $x\in\overline{\Sigma}\cap\Omega$ for which $\partial u(x)$ is a singleton.

\begin{lemma}\label{cusp}
If $x_0\in\Sigma$ and $\partial u(x_0)=\{p_0\}$, then $p_0$ belongs to a boundary face of $\Omega^*$ of dimension at most $n-2$.
\end{lemma}
\begin{proof}
By \eqref{extremes}, $p_0\in\overline{\Omega^*}$. If $p_0\in\Omega^*$, \cite[Proposition~2 and proof of Theorem~3.3]{FK1} makes $u$ smooth near $x_0$, contrary to $x_0\in\overline A$. Suppose instead that $p_0$ lies in an open facet with inward normal $\nu^*$. Translate $(x_0,p_0)$ to $(0,0)$ and subtract the supporting plane, so that $u\ge0$ and $u(0)=0$. For small $\rho>0$, $\Omega^*\cap B_\rho\subset\{y:y\cdot\nu^*\ge0\}.$ Upper semicontinuity at $\partial u(0)=\{0\}$ provides $B_r\Subset\Omega$ with $Du(B_r)\subset\Omega^*\cap B_\rho$ almost everywhere. Thus $D_{\nu^*}u\ge0$ almost everywhere in $B_r$. Absolute continuity on lines and continuity imply $0\le u(t\nu^*)\le u(0)=0$ for $-r<t<0,$ contradicting local strict convexity.
\end{proof}

\begin{proof}[Proof of Theorem~\ref{complete}]
Let $R$ be the set of points $x\in A$ for which $\partial u(x)$ is a segment with both extreme points on open facets. Proposition~\ref{regularity} shows that $R$ is relatively open in $\Sigma$ and that $\Sigma$ is locally smooth at every point of $R$. Hence $E=\Sigma\setminus R$ is relatively closed in $\Sigma$. The stratification theorem~\cite[Theorem~4.1]{AAC} implies
\[
\dim_{\HH}\{x\in\Omega:\dim\partial u(x)\ge2\}\le n-2.
\]
For $x\in E\cap A$ with $\dim\partial u(x)=1$, write $\partial u(x)=[p_1,p_2]$. At least one extreme point of its subdifferential lies on the target skeleton. Indeed by Lemma \ref{middle} we have $\partial u(x)\cap \overline{\Omega^*}=\{p_1, p_2\}.$ If both $p_1, p_2$ are on open facets, then it follows from Proposition~\ref{regularity} that $x\in R.$ If $x\in E\setminus A$, then $u$ is differentiable at $x$. Since $x\in\Sigma$, Lemma~\ref{cusp} implies that $Du(x)$ belongs to a boundary face of $\Omega^*$ of dimension at most $n-2$.

For each relatively open boundary face $F$ with $\dim F\le n-2$, define
\[
P_F=
\bigcup_{\substack{x\in E\\ \dim\partial u(x)\le1}}
\bigl(F\cap\ext\partial u(x)\bigr).
\]
Theorem~\ref{strictness} gives $Dv(p)=x$ whenever $p\in\partial u(x)$ and $x\in\Omega$. The preceding two cases show that every $x\in E$ with $\dim\partial u(x)\le1$ has an extreme subgradient in one of these faces. Hence
\[
E\subset
\{x\in\Omega:\dim\partial u(x)\ge2\}
\cup\bigcup_{\dim F\le n-2}Dv(P_F).
\]

For every $p\in P_F$ and $\alpha\in(0,1)$, Theorem~\ref{estimate} yields constants $C_{p,\alpha},r_{p,\alpha}>0$ such that
\[
|Dv(q)-Dv(p)|
\le C_{p,\alpha}|q-p|^\alpha
\qquad(q\in P_F\cap B_{r_{p,\alpha}}(p)).
\]
Since $P_F$ lies in an affine plane of dimension $\dim F$ and $Dv(P_F)\subset\Omega$ is bounded, Lemma~\ref{dimension} gives $\dim_{\HH}Dv(P_F)\le\dim F\le n-2.$ There are finitely many boundary faces, so the covering above implies $\dim_{\HH}E\le n-2$.
\end{proof}

\begin{corollary}\label{closure}
Under the hypotheses of Theorem~\ref{complete},
\[
\dim_{\HH}(\overline A\cap\Omega)\le n-1.
\]
\end{corollary}
\begin{proof}
The relatively open set $\Sigma\setminus E$ has a countable cover by smooth hypersurface charts. Its Hausdorff dimension is at most $n-1$, whereas $\dim_{\HH}E\le n-2$.
\end{proof}

\section{The free boundary in optimal partial transport}
\label{freeboundaries}

We prove Theorem~\ref{partial} by applying the preceding estimates to the partial transport.

Let $\Omega,\Omega^*\subset\mathbb R^n$ be bounded open sets separated by a hyperplane, and assume that $\Omega^*$ is polyhedral. Fix $0<m<\min\{|\Omega|,|\Omega^*|\}$. Denote the active source and target by $U\subset\Omega$ and $V\subset\Omega^*$, respectively. We write $\mathcal F=\partial U\cap\Omega$ and denote the optimal partial transport plan by $\gamma$.

By \cite[Theorem~4.3 and Corollary~4.5]{CM}, choose the convex potential with quadratic obstacles and set $\Omega_m^*=V\cup\operatorname{int}(\Omega\setminus U)$. Then
\begin{equation}\label{augmented}
(Du)_\#\chi_\Omega=\chi_{\Omega_m^*},
\qquad (Dv)_\#\chi_{\Omega_m^*}=\chi_\Omega,
\end{equation}
and
\[
v(y)=\sup_{x\in\Omega}\{x\cdot y-u(x)\},
\qquad u=v^*\quad\text{in }\Omega.
\]
For a constant $c_m\in\R$, the source obstacle relation is
\begin{equation}\label{obstacle}
u\ge\tfrac12|x|^2-c_m\quad\text{in }\Omega,
\qquad U=\{x\in\Omega:u(x)>\tfrac12|x|^2-c_m\}.
\end{equation}
In particular, $u(x)=\tfrac12|x|^2-c_m$ on $\Omega\setminus U$. The local argument in \cite[Proposition~5.2]{CM} shows that $\partial U\cap\Omega$ is a Lipschitz graph over the separating hyperplane near each of its points. Interchanging source and target yields the same conclusion for $\partial V\cap\Omega^*$. Both sets have zero Lebesgue measure, so the inactive sets may be replaced by their interiors in \eqref{augmented}. For every differentiability point $x\in U$ of $u$, let $y=Du(x)$. By \cite[Corollary~2.4 and Example~2.5]{CM},
\begin{equation}\label{balls}
B_{|x-y|}(y)\cap\Omega\subset U,
\qquad B_{|x-y|}(x)\cap\Omega^*\subset V.
\end{equation}
These inclusions also hold for limits of such pairs.

Let $x\in\mathcal F$. Approaching $x$ through differentiability points on either side of the free boundary, we obtain
\begin{equation}\label{active}
x\in\partial u(x),\qquad
\partial u(x)\cap\overline V\ne\varnothing.
\end{equation}
Moreover,
\begin{equation}\label{activeextremes}
\operatorname{ext}\partial u(x)
\subset\{x\}\cup\overline V.
\end{equation}
Indeed, each extreme point of the subdifferential is a limit of gradients at nearby differentiability points. The inactive gradients converge to $x$, and the active gradients belong to $V$.

We first establish strict convexity for $u$ and the local equation satisfied by $v$ at the active extreme points of its subdifferential.

\begin{lemma}\label{partialgeometry}
The function $u$ is locally strictly convex in $\Omega$. For $x\in\Omega$ and $p\in\partial u(x)$, one has $\partial v(p)=\{x\}$, and $v$ is $C^1$ near $p$. If $x\in\mathcal F$ and $p\ne x$ is an extreme point of $\partial u(x)$, then $p\in\partial\Omega^*$ and, for some $\rho>0$,
\begin{equation}\label{partialequation}
B_\rho(p)\cap\Omega^*\subset V,
\qquad \mu_v=\chi_{\Omega^*}\,dy\quad\text{in }B_\rho(p).
\end{equation}
\end{lemma}

\begin{proof}
The quadratic-obstacle construction \cite[Theorem~4.3 and Corollary~4.5]{CM} also provides a convex function $\bar u$ such that
\[
(D\bar u)_\#\chi_{U\cup\operatorname{int}(\Omega^*\setminus V)}=\chi_{\Omega^*},\qquad \bar u=u\quad\text{in }U.
\]
Theorem~\ref{strictness} makes $u$ locally strictly convex in $U$. An affine segment in $\Omega$ cannot meet $U$ and cannot be contained in $\Omega\setminus U$, where $u=\tfrac12|x|^2-c_m$. Hence $u$ is locally strictly convex in $\Omega$. The argument in Step~5 of the proof of Theorem~\ref{strictness} now gives $\partial v(p)=\{x\}$ and $v\in C^1$ near every $p\in\partial u(x)$, $x\in\Omega$.

Now let $x\in\mathcal F$ and $p\ne x$ be an extreme point of $\partial u(x)$. As in \cite[proof of Lemma~4.1]{CLL}, \cite[Proposition~2]{FK1}, applied to \eqref{augmented}, gives $\partial u(x)\cap V=\varnothing.$ Thus $p\in\partial V$ by \eqref{activeextremes}. The approximation and interior-ball argument in \cite[proof of Lemma~4.2]{CLL} apply in every dimension: they use only \eqref{balls}, monotonicity, and strict convexity. Consequently, $p\in\partial\Omega^*,$ and $B_\rho(p)\cap\Omega^*\subset V$ for some $\rho>0$.

By \eqref{measure}, applied to \eqref{augmented}, and upper semicontinuity at $\partial v(p)=\{x\}\subset\Omega$, we have $\mu_v=\chi_{\Omega_m^*}\,dy$ near $p.$ On $\operatorname{int}(\Omega\setminus U)$, inverse transport gives $Dv(y)=y$ almost everywhere. Continuity and $Dv(p)=x\ne p$ exclude this inactive set from a smaller ball about $p$. Together with $B_\rho(p)\cap\Omega^*\subset V$, this proves \eqref{partialequation}.
\end{proof}

When $\partial u(x)$ is a segment, Theorem~\ref{estimate} now applies to the partial transport.

\begin{lemma}\label{partialestimate}
Let $x\in\mathcal F$ and assume that $\dim\partial u(x)=1$. Then
\[
\partial u(x)=[x,p]\quad\text{for some }p\in\partial\Omega^*.
\]
For every $0<\eta<1$ there are $C,r>0$ such that
\begin{equation}\label{partialgradient}
|Dv(y)-x|\le C|y-p|^{1-\eta}\qquad(|y-p|<r).
\end{equation}
\end{lemma}

\begin{proof}
Since the domains are separated, \eqref{active} and \eqref{activeextremes} give $\partial u(x)=[x,p]$, with $p\in\partial\Omega^*$ by Lemma~\ref{partialgeometry}. By Lemma~\ref{partialgeometry}, Theorem~\ref{estimate} applies to \eqref{augmented}, proving \eqref{partialgradient}.

\end{proof}

The H\"older estimate for $Dv$ also implies that the relative interior of the segment is disjoint from $\overline{\Omega^*}$.

\begin{lemma}\label{partialmiddle}
If $x\in\mathcal F$ and $\partial u(x)=[x,p]$, then
\begin{equation*}
[x,p]\cap\overline{\Omega^*}=\{p\}.
\end{equation*}
\end{lemma}

\begin{proof}
Subtract the supporting affine function from $v$, so that $v\ge0$ and $v=0$ on $[x,p]$. Integrating \eqref{partialgradient}, with $0<\eta<1/n$, yields $v(p+X)\le C|X|^{2-\eta}$. By convexity,
\begin{equation}\label{partialgrowth}
v((1-t)x+tp+X)
\le t\,v(p+X/t)\le C_t|X|^{2-\eta},\qquad 0<t<1,
\end{equation}
where $C_t$ is uniform on compact subintervals of $(0,1)$.

Suppose $y\in(x,p)\cap\overline{\Omega^*}$. Set $e=(p-x)/|p-x|$ and choose $\ell>0$ so that $y+[-2\ell,2\ell]e\subset(x,p)$. On a cylinder about this segment, \eqref{partialgrowth} bounds $v$ by $C|X|^{2-\eta}$, where $X\perp e$. Let $\mathcal T_\rho$ be a smaller cylinder of fixed axial length and transverse radius $\rho$. By convexity,
\[
|\xi\cdot e|\le C\rho^{2-\eta},\qquad
|\xi-(\xi\cdot e)e|\le C\rho^{1-\eta}
\quad(\xi\in\partial v(\mathcal T_\rho)).
\]
Hence
\begin{equation}\label{partialtube}
\mu_v(\mathcal T_\rho)\le C\rho^{n+1-n\eta}.
\end{equation}
On the other hand, $y\in B_{|x-p|}(x)$, so \eqref{balls} implies $V=\Omega^*$ near $y$. By \eqref{cones},
\[
\mu_v(\mathcal T_\rho)\ge|V\cap B_{c\rho}(y)|
=|\Omega^*\cap B_{c\rho}(y)|\ge c_y\rho^n
\]
for small $\rho$. This contradicts \eqref{partialtube}, since $n\eta<1$.
\end{proof}

We next consider segment subdifferentials with an endpoint in the relative interior of a facet of $\Omega^*$.

\begin{proposition}\label{partialfacet}
Suppose $x\in\mathcal F$, $\partial u(x)=[x,p]$, and $p$ lies in the relative interior of a boundary facet of $\Omega^*$. Then $\mathcal F$ is a smooth hypersurface near $x$.
\end{proposition}

\begin{proof}
By Lemmas~\ref{partialgeometry} and~\ref{partialmiddle}, choose $\rho>0$ such that $V\cap B_\rho(p)=\Omega^*\cap B_\rho(p)$ is a half-ball. Upper semicontinuity implies that all active gradients near $x$ lie in $B_{\rho/2}(p)$. Define
\[
u_1(X)=\sup_{y\in\overline{\Omega^*\cap B_\rho(p)}}
\{X\cdot y-v(y)\}.
\]
Lemma~\ref{localization}, \cite[Lemma~7.6, proof of Theorem~7.13 and Remark~7.10]{CM}, and the local duality argument of \cite[proof of Theorem~1.2]{CLL} yield, in a neighborhood of $x$,
\[
u_1\in C^{1,\beta},\qquad
\mu_{u_1}=\chi_U\,dX,\qquad
u=\max\{u_1,\tfrac12|X|^2-c_m\},
\]
for some $\beta\in(0,1)$, with $Du_1(\mathcal F\cap B_r(x))$ contained in the target facet for some sufficiently small $r>0$.

Since $Du_1(x)=p\ne x$, the implicit function theorem gives $\mathcal F\in C^{1,\beta}$ near $x$, with inner normal
\[
\nu_U(X)=\frac{Du_1(X)-X}{|Du_1(X)-X|}.
\]
Let $\nu^*$ be the inner normal of the target facet. Lemma~\ref{partialmiddle} implies
\[
\nu_U(x)\cdot\nu^*
=\frac{(p-x)\cdot\nu^*}{|p-x|}>0.
\]
Thus Proposition~\ref{boundary}(i), together with the normal formula, yields $C^{1,1-\varepsilon}$ regularity of $\mathcal F$. Part~(ii) and the boundary Schauder bootstrap in \cite[Remark~4.4]{CLW3} then give smoothness.
\end{proof}

\begin{proof}[Proof of Theorem~\ref{partial}]
Let $R_{\mathcal F}$ be the set of points covered by Proposition~\ref{partialfacet}. Its local branch description shows that $R_{\mathcal F}$ is relatively open in $\mathcal F$, so $E_{\mathcal F}:=\mathcal F\setminus R_{\mathcal F}$ is relatively closed.

By the stratification theorem~\cite[Theorem~4.1]{AAC},
\[
\dim_{\mathcal H}
\{x\in\mathcal F:\dim\partial u(x)\ge2\}\le n-2.
\]
Every other point of $E_{\mathcal F}$ has $\partial u(x)=[x,p]$, where $p$ belongs to a target face of dimension at most $n-2$. At each such point $Dv(p)=x$, and \eqref{partialgradient} holds for every $\eta\in(0,1)$. Applying Lemma~\ref{dimension} on each of the finitely many target faces, we find that these source points also have Hausdorff dimension at most $n-2$. This proves the theorem.
\end{proof}
\appendix
\section{Regularity at a flat target facet}\label{flat}

Let $\Omega,\Omega^*\subset\R^n$ be local source and target domains, and let $u,v$ be dual convex potentials with $Du(x_0)=p_0$ and $Dv(p_0)=x_0$. Suppose $\nu_\Omega(x_0)\cdot\nu_{\Omega^*}(p_0)>0.$ After translations, dual affine changes of coordinates, and subtraction of supporting planes, we may assume
\[
x_0=p_0=0,\qquad u(0)=v(0)=0,\qquad Du(0)=Dv(0)=0,
\qquad \nu_\Omega(0)=\nu_{\Omega^*}(0)=e_n.
\]
Suppose that, near $0$,
\[
\Omega^*=\{y_n>0\},\qquad
\Omega=\{x_n>\rho(x')\},
\]
where $\rho\in C^{1,\beta}$ for some $\beta\in(0,1)$ and $\rho(0)=D\rho(0)=0$.

Assume that $u$ is the Legendre transform of $v$ restricted to a closed local target cap, that $u\in C^{1,\beta}(\overline\Omega)$ and $v\in C^1$ locally, and that
\[
\mu_u=\chi_\Omega\,dx,\qquad
\mu_v=\chi_{\Omega^*}\,dy,\qquad
Du(\partial\Omega)\subset\partial\Omega^*
\]
in the corresponding neighborhoods. Suppose further that
\[
S_h^c[v](y)\subset B_r(0)
\qquad
(y\in\overline{\Omega^*}\cap B_{\delta_0}(0),\
0<h\le h_0)
\]
for some $r,\delta_0,h_0>0$.

As in \cite[Proposition~6.1]{CLL}, the arguments in \cite[Sections~3 and~4]{CLW3} yield the following result.

\begin{proposition}\label{boundary}
Under the above assumptions:
\begin{enumerate}
\item[(i)] For every $\varepsilon\in(0,1)$,
\[
u(x)\le C_\varepsilon|x|^{2-\varepsilon}
\qquad(x\in\overline\Omega\text{ near }0),
\]
and $u$ satisfies boundary $C^{1,1-\varepsilon}$ estimates.

\item[(ii)] If $\beta$ is sufficiently close to $1$, then
\[
u\in C^{2,\alpha}(B_{\bar r}(0)\cap\overline\Omega)
\]
for some $\alpha\in(0,1)$ and $\bar r>0$.
\end{enumerate}
\end{proposition}

\begin{proof}
Fix a centered section $S_{h_0}^c[v](0)\Subset B_r(0)$. Its Monge--Amp\`ere density is $\chi_{\{y_n>0\}}$. The directional Pogorelov estimate \cite[Corollary~1.1]{C96}, applied on this section, yields
\[
v(y',0)\le C|y'|^2,\qquad u(x)\ge c|x'|^2\quad(x\in\overline\Omega\text{ near }0),
\]
where the second inequality follows by duality. The normal estimate in \cite[Lemma~3.2]{CLW3} and the uniform density argument for a $C^{1,\beta}$ source graph in \cite[Lemma~5.2]{CL3} give
\[
u(te_n)\le C_\varepsilon|t|^{2-\varepsilon},\qquad |S_h^c[u](0)\cap\Omega|\ge\delta|S_h^c[u](0)|
\]
for small $|t|$ and $h$, with $\delta>0$ independent of $h$. The tangential iteration \cite[Lemma~3.4]{CLW3} and \cite[Corollary~3.2]{CLW3} prove (i). For $\beta$ sufficiently close to $1$, the comparison estimate \cite[Lemma~4.1]{CLW3} and the iteration in \cite[Section~4]{CLW3} prove (ii), with the target boundary function identically zero. These local arguments apply in $\R^n$.
\end{proof}

\section*{Acknowledgments}
This paper grew out of a research project initiated by the authors in 2022. The two conjectures were formulated entirely by the first author and stated in \cite{CLL}. The authors developed the proof strategy, including the use of smooth approximations of the inverse potential to avoid decomposing the singular part of the Hessian measure. ChatGPT (OpenAI) was used primarily to assist with some tedious computations involved in deriving the monotonicity formula in Section~\ref{secmon}.

\end{document}